\documentclass[leqno, 12pt]{article}
\usepackage{amsfonts,amsmath,amssymb,amsthm}
\usepackage{mathrsfs}
\usepackage[T1]{fontenc}
\usepackage[english]{babel}
\usepackage[all]{xy}
\usepackage{verbatim}
\usepackage{bbm}
\usepackage{makeidx}

\newcommand{\Z}{\mathbb{Z}}
\newcommand{\R}{\mathbb{R}}

\renewcommand{\O}{\mathcal{O}}

\newcommand{\vuoto}{\varnothing}

\newcommand{\iso}{\cong}

\newcommand{\dual}{^{\vee}}

\usepackage{amscd}
\usepackage{txfonts}
\usepackage[cm]{fullpage}
\newtheorem{thr}{Theorem}[section]
\newtheorem{lmm}[thr]{Lemma}
\newtheorem{prp}[thr]{Proposition}
\newtheorem{crl}[thr]{Corollary}

\theoremstyle{definition}
\theoremstyle{remark}\newtheorem{rmk}[thr]{Remark}
\theoremstyle{remark}
\theoremstyle{remark}

\title{Clemens-Schmid exact sequence in characteristic $p$}
\author{B. Chiarellotto, N. Tsuzuki}

\makeindex

\begin{document}

\maketitle
\abstract{  For a semistable family of varieties over a curve in characteristic $p$, we prove the existence 
of a "Clemens-Schmid type" long exact sequence for the $p$-adic cohomology. 
The cohomology groups appearing  in such a long exact sequence are  defined locally.}

\let\thefootnote\relax

\section{Introduction}
 
Let $\Delta$ denote an open  disk around $0$ in the complex plane. 
Let $X$ be a smooth complex variety which is a K\"ahler manifold. 
Consider a semi-stable degeneration $\pi : X \rightarrow \Delta$, {\it i.e.}, 
a holomorphic, proper and flat map of relative dimension $n$ 
such that $\pi$ is smooth outside  the fiber $X_0 = \pi^{-1}(0)$ which is, furthemore, assumed 
to be a strict normal crossing divisor 
(in other words $X_0= \sum X_{0,i}$ is a sum of irreducible components $X_{0,i}$ of $X_0$ meeting transversally 
and each  $X_{0,i}$ is smooth). 
In this situation, for any $m$, one can  associate a limit cohomology  $H^m_{\mathrm{lim}}$ for $X_0$ 
(see \cite{mo} or \cite{st}). 
This $H^m_{\mathrm{lim}}$ is endowed with a nilpotent monodromy operator $N$ 
and a weight  filtration from a mixed Hodge structure. 
One has  $H^m(X_t) \simeq H^m_{\mathrm{lim}}$ as vector spaces for $t \ne 0$ where $X_t = \pi^{-1}(t)$; 
moreover, a topological argument shows  that 
$H^m(X) \simeq H^m(X_0):= H^m$ and $H_m(X) \simeq H_m(X_0):=H_m$ as well. 
By  $i$ we will indicate the inclusion $X_t \rightarrow X$. 
Then it is possible to define the Clemens-Schmid exact sequence (respecting MHS) \cite{cle} Chap. 1, 3.7 (see also \cite{mo}):

\begin{equation*} 
\cdots \rightarrow  H_{2n+2-m} \xrightarrow{\alpha} H^m  \xrightarrow{i^{\ast}} H^m_{\mathrm{lim}} \xrightarrow{N} 
H^m_{\mathrm{lim}} \xrightarrow{\beta} H_{2n-m} \xrightarrow{\alpha}  H^{m+2} \rightarrow \cdots, 
\end{equation*}

\noindent 
where the maps $\alpha$  are the natural maps arising from the Poincar\'e duality for $X_0$ considered as closed 
in the smooth variety $X$, and the maps $\beta$ 
are again  obtained via Poincar\'e duality for $H^m(X_t) \simeq H_{2n-m}(X_t)$ 
 composed with the natural map $H_{2n-m}(X_t) \rightarrow  H_{2n-m}(X)$ dual to $i^{\ast}$.

In order to  prove the exactness of this long  sequence one needs more than a  topological argument 
which  connects a global definition of the cohomology of $X$ with support in $X_0$ to a sequence involving the cohomology 
of the special and generic fibers. Indeed one  also needs a ``{\it weights}" argument. 
In fact one has to use the fact that the sequence respects the  weight filtrations of mixed Hodge structures of the vector spaces 
involved and moreover that the weight and monodromy  filtrations  on $H_{\mathrm{lim}}^n$ coincide. 
We recall, also, that the structure of the limit cohomology has been considered 
in the framework of  log-geometry (see \cite{il}).

\bigskip

In this article we deal with the analogous  situation in  characteristic  $p\! >0$. Namely, we  consider the following morphism

\begin{equation*}
	f: X \rightarrow C \ 
	\end{equation*}

\noindent 
over a finite field $k$ of characteristic $p$, where $X$ is a smooth variety of dimension $n+1$, $C$ is a smooth curve and 
$f$ is a proper and flat  morphism. We suppose  that for  a $k$-rational point $s$  of $C$, the fiber at $s$ of $f$, $X_s$, 
is a normal crossing divisor (NCD for short) and $f$ is smooth outside $X_s$. 
Let $\cal V$ be a complete and absolutely unramified 
discrete valuation ring of mixed characteristic whose residue (resp. fraction) field is $k$ (resp.  $K$). 
Then we prove the existence of a Clemens-Schmid sequence:
\begin{eqnarray} \label{sequenceA}
\\
\cdots &\rightarrow& H^m_{\rm rig}(X_s) \xrightarrow{\gamma} 
H^m_{\rm log-crys}((X_s,M_s)/{\cal V}^\times) \otimes K  \xrightarrow{N_m} 
H^m_{\rm log-crys}((X_s,M_s)/{\cal V}^\times) \otimes K(-1) \xrightarrow{\delta} H^{m+2}_{X_s,{\rm rig}}(X) \nonumber \\
&\xrightarrow{\alpha}& H^{m+2}_{\rm rig}(X_s) \rightarrow \cdots,  \nonumber
\end{eqnarray}
where $M$ is a log-structure on $X$ associated to the NCD $X_s$ (in \'etale topology), 
$M_s$ is the fiber of $M$ at $s$, and $(a)$ means the $a$-th Tate twist of Frobenius structure.

The role of the  limit cohomology will be played by the log-crystalline cohomology of the log-scheme $X_s$ 
endowed with the log-structure $M_s$ induced by the log-structure $M$ of $X$ given by the NCD $X_s$ itself.
We denote this limit cohomology by $H^m_{\rm log-crys}((X_s,M_s)/{\cal V}^\times)$ 
($\mathcal V^\times$ is endowed with the log-structure on $\mathcal V$ associated to $1 \mapsto 0$). 
We will  then consider  the  cohomology of the special fiber $X_s$ without any structure: 
and here we will apply rigid cohomology, $ H^m_{\mathrm{rig}}(X_s)$.
We now need to replace the "trascendental" topological argument used  to construct such a sequence.  
The underlying idea is that the bridge between the local ({\it i.e.}, based 
on the special fiber) and the global ({\it i.e.}, based on the whole geometric 
datum) will be given by two different definitions 
of the  cohomology of $X$ with support  in $X_s$, $H^m_{X_s, \mathrm{rig}}(X)$. 
Moreover  we will link it to the cohomology of the open complement of $X_s$ in $X$ 
which  will be understood in the framework 
of a generalized log-convergent cohomology theory introduced by Shiho. Hence  we will  have a long exact sequence

\begin{equation} \label{seqsupp}
\cdots \longrightarrow H^m_{X_s, \mathrm{rig}}(X)\,
\displaystyle{\mathop{\longrightarrow}^\alpha}\, H^m_{\mathrm{rig}}(X_s)\, 
\longrightarrow\, H^m_{ \mathrm{log-conv}}((X_s, M_s)/\mathcal V)\, \longrightarrow \cdots. 
\end{equation}

 \noindent
 Furthermore  the (absolute) log-convergent cohomology groups, $H^m_{\mathrm{log-conv}}((X_s, M_s)/\mathcal V)$, 
 will be linked to the limit cohomology of the special fiber via the following long exact sequence

 \begin{equation} \label{rel6}
\dots \rightarrow H^m_{\rm log-conv}((X_s, M_s)/{\cal V}) \rightarrow 
H^m_{\rm log-crys}((X_s,M_s)/{\cal V}^\times) \otimes K \xrightarrow{N_m} 
H^m_{\rm log-crys}((X_s,M_s)/{\cal V}^\times) \otimes K(-1) \rightarrow \cdots,
\end{equation}

\noindent 
because we will interpretate the monodromy in terms of Gauss-Manin connection. 
Then we can merge the two sequences \eqref{seqsupp} and \eqref{rel6} to obtain the 
$p$-adic analogue of the Clemens-Schmid sequence \eqref{sequenceA} as in the classical case discussed in \cite{cle}.

For the exactness of the merged  sequence we note that the log-crystalline cohomology of $X_s$ admits  a weight  structure 
(coming from the Frobenius action) and the existence of a monodromy operator. Following the work of Shiho,  
we will insert our  cohomology into a family.  
Moreover to such a family
we will associate a differential operator having  a regular  singular point at $s$ and endowed with a Frobenius structure. 
We will re-interpret the monodromy operator in terms of residue of the differential operator at $s$.  
In this differential setting  the equivalence between the  monodromy and weight filtration (given by Frobenius) 
has  been proved by Crew \cite{crew} 10.8. Therefore, we can prove the exactness of 
the Clemens-Schmid type sequence by weight arguments. 

We also have Poincar\'e duality in the rigid setting obtaining
$$ 
H^{m+2}_{X_s,{\rm rig}}(X)\iso H_{c,{\rm rig}}^{2n+2-m-2}(X_s)\dual(-\mathrm{dim}\, X) 
\iso H_{2n-m}^{\rm rig}(X_s)(-\mathrm{dim}\, X) 
$$ 
\cite{ls}, \cite{Pet:Cla03}, \cite{Ber:Dua97}: hence \eqref{sequenceA} 
is the complete analogue of the classical Clemens-Schmid exact sequence.

In this paper we will prove the exactness of the Clemens-Schmid type sequence when  $k$ is a finite field 
because we would like to avoid the difficulty of building up the relative theory. 
We believe, however, that it is possible to remove the finiteness  hypothesis. 
We do not treat the problem of the exactness for a  proper semistable family defined over ${\rm Spec}\, k[[t]]$.

It is tempting  to try to see our procedure  along the lines of Levine's article on  motivic tubular neighborhoods  \cite{le}. 
In that article  topological methods were replaced by the notion   of tubular neighborhoods. In our  $p$-adic realization,  
this corresponds  to the use of the tubes in characteristic 0 for our varieties in characteristic $p$. 
With respect to his approach, in our ``realization''  we have the advantage of a  weight filtration 
which can be  compared  to the monodromy  filtration and we can prove the exactness of our Clemens-Schmid sequence. 
We also mention that Nakkajima studied the kernel of the monodromy operator 
in crystalline settings in \cite{na}, Sect. 6. 

Here is an outline of this paper.  After  establishing our notations and conventions, we will show in section 3, 
how Shiho's theory of relative log-cohomology (log-crystalline, log-convergent and log-analytic) can be used in our setting. 
In particular we will understand the log-crystalline cohomology of the special fiber $X_s$ as a fiber at $s$ 
of the sheaves of relative log-cohomology on $C$ which are endowed with a convergent connection, 
whose residue at  $s$ will be the monodromy operator. In the section 4 we will construct the sequence.  
We will  show  how  the global long exact sequence should  be defined using ``local" objects via tubes:  
this will replace the topological methods in the classical case. 
To do that we will need to compare and link several cohomology theories: 
these  results will be obtained by choosing a good embedding system.  
In fact, in \cite{shiho:fund2} 2.2.4, Shiho gave a functorial definition for  a log tubular neighborhood 
$]X[^{\mathrm{log}}_{\mathscr P}$ arising from a (not exact, in general) closed immersion $X \rightarrow \mathscr P$ 
of  the log $k$-scheme $X$  to the formal log $\mathcal V$-scheme $\mathscr P$ (under some assumptions), 
by which he is able to link log-crystalline cohomology with log-convergent (log-analytic) cohomology. 
For our purpose the factorial construction of log tubular neighborhoods is not enough, 
and we need a good embedding system in the level of log formal schemes for certain \'etale hypercoverings 
which enjoys smoothness, log-smoothness and exactness properties in order to link log-convergent 
(log-analytic) cohomology with rigid cohomology. 
This is done in Propositions \ref{emb} and \ref{emb2}. 
In section 5 we will  prove the last ingredient for the exactness of the Clemens-Schmid sequence: 
the  monodromy filtration coincides with the weight filtration for the log-crystalline cohomology of $X_s$. 
This will be proved using the theory of the third section: namely the fact that we  may view this cohomology as a special fiber at $s$ 
of a module endowed with a log-connection and a Frobenius structure  on the curve (with a log-structure given by the special point) 
and with  monodromy  given by the residue of this differential module at that special point (hence the monodromy is  unipotent). 
In this sense  we will use Crew's  results on the equivalence of  two filtrations \cite{crew} \S10, 
which, in turn, was an adaptation of Deligne's methods for the \'etale setting \cite{de} 1.8.4.

\paragraph*{Acknowledgments}

We would like to thank A. Shiho and A. Lauder.  Lauder's article \cite{laud} was the starting point for this article. 
The first author was supported by Cariparo Eccellenza Grant "Differential methods in algebra, geometry and arithmetic" 
and was also supported by a JSPS grant.  
The second author was supported by a JSPS  Grant-in-Aid for research (B) 22340001. 
We thank Frank Sullivan. 

\section{Notation and setting} 

In this paper  we will indicate by $k$ a perfect field of characteristic $p>0$, if not otherwise indicated.
For simplicity we denote by  $\cal V$  the ring of Witt vectors of  $k$, $K$ its fraction field. 
The Frobenius is denoted by $\sigma$. We put ${\cal V}_i = {\cal V}/p^i{\cal V}$. 
Of course one could have taken for $\cal V$ any complete discrete valuation 
ring with residue field $k$ and ${\rm Frac}( {\cal V})=K$:  
but all the  $K$-cohomology groups we are going to consider 
in this case will be defined by tensoring with $K$ the cohomology groups defined 
over the fraction field of the ring of  Witt vectors of $k$. 
Hence the ramification does not trouble our constructions and results.  

We recall that a divisor $Z \subset Y$ of a Noetherian scheme is said to be a strict normal crossing divisor (SNCD) 
if  $Z$ is a reduced scheme and, if $Z_i , i \in J$ 
are the irreducible components of $Z$, then, for any $I \subseteq J$ (which might be empty) 
the intersection $Z_I= \cap _{i \in I}\, Z_i$ 
is a regular scheme of codimension 
the number of elements of $I$ (or it may be empty). Moreover $Y$ is said to be a normal crossing divisor (NCD) 
if, \'etale locally on $Y$, it is an SNCD.

We consider the following morphism

\begin{equation*} 
	f: X \rightarrow C 
\end{equation*}

\noindent
over a field $k$, where $X$ is a smooth variety of dimension $n+1$,  $C$ is a smooth curve and $f$ is a proper and flat morphism. 
We suppose that, for a $k$-rational point $s$ of $C$,  the fiber at $s$ of $f$, $X_s$, is an NCD  
in $X$ and $f$ is smooth on  $X \setminus X_s$. 
We use  $(X, M)$ to denote  the scheme  $X$ endowed with the log-structure given by the NCD, $X_s$, 
while $(C, N)$ denotes  the curve $C$ endowed with the log-structure $N$ given by $s$ (all for the \'etale topology). 
The induced  map $f : (X, M) \rightarrow (C, N)$ is log-smooth and proper. Then $s^{\times}$ is a log-point given 
by the $k$-rational point $s$ of $(C, N)$, {\it i.e.}, the induced log-structure given 
by the closed immersion $s \rightarrow C$ from $(C, N)$.
We refer to such a situation by a cartesian diagram

\begin{equation} \label{diag1}
\xymatrix{
(X_{s}, M_s)  \ar[d] \ar[r]&  (X, M)\ar[d]^{f}\\
  s^{\times} \ar[r] & (C, N).}
\end{equation}

\noindent
By $\mathcal V^\times$ we indicate $\mathcal V$ endowed with the log-structure associated to $\mathbb N \ni 1 \mapsto 0 \in \cal V$. 
Again $\mathcal V$ will indicate $\mathcal V$ endowed with the trivial log-structure. 

If we say a property is satisfied by a simplicial (formal) scheme (resp. a morphism 
of simplicial (formal) schemes), then we mean it is  satisfied 
at each level of the simplicial (formal) scheme (resp. the morphism 
of simplicial (formal) schemes). 

By $[A^{\bullet} \rightarrow B^{\bullet}]$, 
we understand the simple complex  associated  to the   double one  (also in the simplicial setting) 
for complexes $A^{\bullet}$, $B^{\bullet}$.

\section{Relative Cohomology}
Because in \eqref{diag1},  $C$ was  a smooth  curve over $k$ then it admits a smooth lifting ${C}_{\cal V}$ 
over $\cal V$ (7.4 III, SGA1): 
we indicate by $\mathscr C$ its completion along the special fiber $C$. 
Let us fix a lift $\hat s$ of $s$ in $C_{\cal V}$ and a section $t$  as a local coordinate of $\hat s$ in $C_{\cal V}$ over $\cal V$. 
$\hat s$ and $t$ also denote a lift of $s$ in $\mathscr C$ and a local coordinate 
 of $\hat s$ in $\mathscr C$ over $\cal V$, respectively. 
Then $1 \mapsto t$ defines a log-structure on $\mathscr C$ 
and we indicate it by ${\mathscr N}$. 
After shrinking $C$ it is also possible to endow $(\mathscr C, {\mathscr N})$ 
with a lift $\sigma_{\mathscr C}$ of Frobenius which is compatible with the Frobenius $\sigma$ on $\cal V$. 
We then have a sequence of exact closed immersions of log-schemes

\begin{equation} \label{map2bis} s^{\times}  \rightarrow (C, N) \rightarrow (\mathscr C, {\mathscr N}). 
\end{equation}

\noindent 
The log-scheme $(C, N)$ in \eqref{map2bis} is log-smooth over $k$, where $k$ is endowed with the trivial log-structure, 
and the formal log-scheme $(\mathscr C, {\mathscr N})$ in \eqref{map2bis} is log-smooth over $\mathcal V$. 
We will denote the reduction of $(\mathscr C, \mathscr N)$ 
modulo $p^i$ by $({\mathscr C}_i, {\mathscr N}_i)$. 
We will indicate by ${\mathscr C}_{K}$ (resp. $({\mathscr C}_{K}, {\mathscr N}_K))$ the rigid analytic space associated 
to the generic fiber of $\mathscr C$ (resp. with the log-structure ${\mathscr N}_K$ induced by $\mathscr N$), and denotes by $\hat{s}_K$ 
a point of $\mathscr C_K$ defined by $t = 0$.

As in section 2 we may induce on $X_{s}$ the log-structure of $(X,M)$ and we refer to it as $(X_s, M_s)$. 
Again  $(X, M)$ is log-smooth over $k$ endowed with the trivial log-structure 
while $(X_s, M_s)$ is log-smooth over $s^{\times}$. 
Then we have the following diagram:	
 
\begin{equation} \label{diag2}
\xymatrix{
(X_{s},M_s)  \ar[d]^{f_{s}} \ar[r]&    (X,M)\ar[d]^{f} &\\
  s ^{\times} \ar[r] & (C, N) \ar[r] &(\mathscr C, {\mathscr N}). }
\end{equation}
\noindent 
In this setting Kato and, later, Shiho (\cite{Ka:log}, \cite{shiho:rel1}, \cite{shiho:rel2}) 
were able to define the {\it log-crystalline cohomology sheaves of} 
$(X,M)/(C, N)$ {\it with respect to} $(\mathscr C, \mathscr N)$, 
and we will indicate them by 

\begin{equation*}
	\R^{m }f_{(X,M) / (\mathscr C, {\mathscr N}), {\rm crys} \ast } ({\O}_{\mathrm{crys}, X,K}).
\end{equation*}

\noindent
In this paper we will only work with the trivial log-isocrystal ${\O}_{\mathrm{crys}, X,K}$. 
In order to define the log-crystalline cohomology one needs to fix a Hyodo-Kato 
embedding system $({\mathscr P}_{\bullet}, {\mathscr M}_{\bullet})$ 
of an \'etale hypercovering $(X_\bullet, M_\bullet)$ of $(X, M)$ \cite{hy-ka} 2.18 
(see the notion of simplicial (formal) schemes and \'etale hypercovering in \cite{ct}, \cite{ts}):
\begin{equation}
\label{HKsystem}
\xymatrix{
 (X_{s,\bullet}, M_{s,\bullet}) \ar[d]^{\theta_s} \ar[r]& (X_{\bullet}, M_{\bullet}) \ar[d]^{\theta} \ar[r]^{i_{\bullet}}
 &  ({\mathscr P}_{\bullet}, {\mathscr M}_{\bullet})  \ar[d]^g   \\
(X_s, M_s) \ar[r] \ar[d]^{f_s} & (X,M) \ar[d]^f &   \ar[d] \\
 s^{\times} \ar[r] &(C, N) \ar[r] & (\mathscr C, {\mathscr N})} 
\end{equation}
where 
\begin{list}{}{}
\item[(i)] $\theta : X_\bullet \rightarrow X$ is an \'etale hypercovering such that the inverse image of $X_s$ 
is an SNCD in $X_\bullet$, $M_\bullet$ is the log-structure on $X_\bullet$ 
induced by $M$, 
and $\theta_s : (X_{s, \bullet}, M_{s, \bullet}) \rightarrow (X_s, M_s)$ is an induced \'etale hypercovering 
by base change;
\item[(ii)] $(\mathscr P_\bullet, {\mathscr M}_{\bullet})$ is a simplicial formal log-scheme separated of finite type 
over $(\mathscr P_\bullet, {\mathscr M}_{\bullet})$ with Frobenius endomorphism which extends 
the Frobenius endomorphism $\sigma_{\mathscr C}$ on $\mathscr C$ 
such that the log-structure ${\mathscr M}_{\bullet}$ comes from that on the Zariski topology of $\mathscr P_\bullet$, 
and that $(\mathscr P_{\bullet}, {\mathscr M}_{\bullet})$ 
is log-smooth over $(\mathscr C, {\mathscr N})$;
\item[(iii)] $i_\bullet $ is a closed immersion of simplicial formal log-schemes (not necessary exact). 
\end{list}

\noindent
In this situation the  {\it crystalline complex} 
$C_{(X_\bullet, M_\bullet)/(\mathscr C_i, \mathscr N_i)}$ on the Zariski site on $X_\bullet$ 
can be defined by the logarithmic de Rham complex of the log-PD-envelope of the closed immersion 
$i_\bullet : (X_\bullet, M_\bullet) \rightarrow 
(\mathscr P_\bullet, \mathscr M_\bullet)$ over $(\mathscr C_i, \mathscr N_i)$ (see \cite{Ka:log} 6.2, 6.4). 
Although the log-structures are defined on the \'etale sites, it is a complex of coherent sheaves 
on the Zariski site on $X_\bullet$ by our choice of the \'etale hypercovering and the embedding system. 
If $f_{(X, M)/(\mathscr C, \mathscr N), {\rm crys}}$ is the natural composite 
$$
      \left((X, M)/(\mathscr C, {\mathscr N})\right)_{\mathrm{crys}}^{\mathrm{log}}\, \longrightarrow\, 
      X_{\mathrm{Zar}}\, \longrightarrow\, 
      {\mathscr C}_{\mathrm{Zar}}
$$
where $\left((X,M)/(\mathscr C, {\mathscr N})\right)_{\mathrm{crys}}^{\mathrm{log}}$ is the log-crystalline site of 
$(X,M)/(\mathscr C, {\mathscr N})$ and $X_{\mathrm{Zar}}$ 
(resp. ${\mathscr C}_{\mathrm{Zar}}$) 
is the Zariski site of $X$ (resp. $\mathscr C$), then the log-crystalline cohomology is calculated by 
$$
     \R f_{(X,M) / (\mathscr C, {\mathscr N}), {\rm crys} \ast} ({\O}_{\mathrm{crys}, X,K}) \iso
      K \otimes \mathbb R\varprojlim_i\, 
     \R(f\theta)_\ast C_{(X_\bullet, M_\bullet)/(\mathscr C_i, \mathscr N_i)}. 
$$
Since $f$ is proper and log-smooth, Shiho proved the following proposition (claim in the proof of \cite{shiho:rel1} 1.15). 
See also \cite{BO78} B.9 for $\mathbb R\varprojlim$. 

\begin{prp}\label{consist} Under the notation above, 
\begin{list}{}{}
\item[\mbox{\rm (1)}] $\R(f\theta)_\ast C_{(X_\bullet, M_\bullet)/(\mathscr C_i, \mathscr N_i)} 
\otimes^{\mathbb L}_{\O_{\mathscr C_i}} \O_{\mathscr C_{i-1}}\, \cong\, 
\R(f\theta)_\ast C_{(X_\bullet, M_\bullet)/(\mathscr C_{i-1}, \mathscr N_{i-1})}$ for any $i$. 
\item[\mbox{\rm (2)}] $\R(f\theta)_\ast C_{(X_\bullet, M_\bullet)/(\mathscr C_i, \mathscr N_i)}$ 
is bounded and has finitely generated cohomologies for any $i$.
\end{list}
In particular, the projective system $\{ \R^m(f\theta)_\ast C_{(X_\bullet, M_\bullet)/(\mathscr C_i, \mathscr N_i)} \}_i$ 
satisfies the Mittag-Laffler condition for each $m$. 
\end{prp}

\vspace*{3mm}

\begin{rmk} In section 4 we construct an embedding system such that $i_\bullet$ 
is an exact closed immersion. In this section we do not need such a special embedding system.
\end{rmk}

We recall the notion of iso-coherent sheaves in \cite{shiho:rel1} 1.2. 
The category of iso-coherent sheaves on the Zariski site of a $p$-adic formal scheme 
$\mathcal S$ 
is the category such that an object has a form $\mathcal F \otimes \mathbb Q$ 
for a coherent $\O_{\mathcal S}$-module $\mathcal F$ 
and the  group of homomorphisms is given by 
$\mathrm{Hom}({\cal F}\otimes \mathbb Q, {\cal G}\otimes \mathbb Q) 
= \mathrm{Hom}_{\mathrm{Coh}}({\cal F}, {\cal G}) \otimes \mathbb Q$, where 
$\mathrm{Hom}_{\mathrm{Coh}}({\cal F}, {\cal G})$ 
is the abelian group of homomorphisms as coherent sheaves. 

\begin{thr}\label{relcrys1} Under these hypotheses on $(X, M)$, 
$\R f_{(X,M) / (\mathscr C, {\mathscr N}), {\rm crys} \ast } ({\O}_{\mathrm{crys}, X,K})$ 
is a perfect complex of iso-coherent sheaves on $\mathscr C$. 
Moreover, the iso-coherent cohomology sheaf 
$\R^mf_{(X,M)/(\mathscr C, {\mathscr N}), {\rm crys}\ast}({\O}_{\mathrm{crys}, X,K})$ 
admits a Frobenius structure for each $m$.
\end{thr}

\begin{proof} The perfectness follows from \cite{shiho:rel1} 1.16.  Since $f$ is of Cartier type, \cite{hy-ka} 2.24  provides 
the Frobenius structure.
\end{proof}

Moreover Shiho (\cite{shiho:rel1} 1.19) was also able, always in our setting \eqref{diag2}, 
to define a base change theorem (see also \cite{Ka:log} 6.10). 
In fact one can complete the diagram  \eqref{diag2} as

\begin{equation} \label{diag3}
\xymatrix{
(X_{s},M_s)  \ar[d]^{\iota} \ar[r]^{f_{s}}&  s^{\times}  \ar[d] ^{\iota} \ar[r]& \hat{s}^\times = {\cal V}^\times \ar[d]^{\hat{\iota}}\\
 (X,M) \ar[r]^f & (C, N) \ar[r] &(\mathscr C, {\mathscr N}),}
\end{equation}

\noindent 
where we indicate the morphism defined by $t \rightarrow 0$ 
by $\hat{\iota}$ and identify $\hat{s}^\times$ with $\cal V^\times$. Note that all the squares are cartesian. 
The $\iota$'s and $\hat{\iota}$ are exact closed immersions and 
$f$ and $f_s$ are proper and log-smooth. Then, following  \cite{shiho:rel1}1.19:
\begin{thr}\label{basechange1} As in  (\ref{diag3}) we  have an isomorphism
$$ 
\mathbb L\hat{\iota}^{\ast}\R f_{(X,M)/(\mathscr C, {\mathscr N}), {\rm crys} \ast}({\O}_{\mathrm{crys}, X,K})  
\iso
{\R f_{s}}_{(X_{s}, M_s) / {{\cal V}^\times}, {\rm crys} \ast } ({\O}_{\mathrm{crys}, X_s, K})
$$
in the derived category of perfect $K$-complexes.
\end{thr}

 \medskip
 $$ \ast  \ast \ast$$

We will need not only log-crystalline cohomology, 
but also  the entire apparatus developed by Shiho in his work: 
namely two other relative cohomology theories: the log-convergent and the 
log-analytic cohomologies.
In our setting, Shiho \cite{shiho:rel1}, \cite{shiho:rel2} was able to introduce 
the {\it relative m-th  log-convergent cohomology of} 
$(X,M)/(C, N)$ with respect to $(\mathscr C, {\mathscr N})$ which are indicated by 
\begin{equation*} 
	\R^{m}f_{(X,M)/(\mathscr C, {\mathscr N}), {\rm conv} \ast }({\O}_{\mathrm{conv}, X,K}).
\end{equation*}
Here ${\O}_{{\rm conv},X,K}$ is the trivial convergent isocrystal. Let us denote the log-tube of the closed immersion 
$i_\bullet : (X_\bullet, M_\bullet) \rightarrow (\mathscr P_\bullet, \mathscr M_\bullet)$, {\it i.e.}, 
the usual tube $]X_\bullet[_{{\mathscr P}_\bullet^{\mathrm{ex}}}$ of the exactification ${\mathscr P}_\bullet^{\mathrm{ex}}$ 
of $i_\bullet$ Zariski locally on ${\mathscr P}_\bullet$, 
by $]X_\bullet[_{\mathscr P_\bullet}^{\mathrm{log}}$, 
and let $\mathrm{sp} : ]X_\bullet[_{\mathscr P_\bullet}^{\mathrm{log}} \rightarrow \widehat{\mathscr P}_\bullet$ 
be the specialization map where ${\widehat{\mathscr P}}_\bullet$ is 
the completion of ${\mathscr P}_\bullet$ along $X_\bullet$ \cite{shiho:fund2} 2.2.4, \cite{shiho:rel1} 2.19. 
Then the log-convergent cohomology is calculated by the logarithmic de Rham complex 
$$
\Omega^{\bullet}_{{] {{ X}_{\bullet}}[^{\rm log}_{{{\mathscr P}}_{\bullet}}}/{{\mathscr C}_{K}}} \! \!  \! 
< {{\mathscr M}_\bullet}/ {\mathscr N}>
$$

\noindent
of the simplicial rigid analytic space $]X_\bullet[_{\mathscr P_\bullet}^{\mathrm{log}}$ over 
$\mathscr C_K = ]C[_{\mathscr C} = ]C[_{\mathscr C}^{\mathrm{log}}$ 
(because the closed immersion $(C, N) \rightarrow (\mathscr C, \mathscr N)$ is exact) 
\cite{shiho:rel1} 2.34: 
$$
     \R f_{(X,M)/(\mathscr C, {\mathscr N}), {\rm conv} \ast }({\O}_{\mathrm{conv}, X,K}) \iso 
     \R (f\theta)_\ast \mathrm{sp}_\ast \Omega^{\bullet}_{{] {{ X}_{\bullet}}[^{\rm log}_{{{\mathscr P}}_{\bullet}}}/]{{\mathscr C}_{K}}} \! \!  \! 
< {{\mathscr M}_\bullet}/ {\mathscr N}>.
$$
\noindent 
Then there is a canonical comparison morphism 
\begin{equation}
\label{crysconv}
      \mathrm{sp}_\ast \Omega^{\bullet}_{{] {{ X}_{\bullet}}[^{\rm log}_{{{\mathscr P}}_{\bullet}}}/{{\mathscr C}_{K}}} \! \!  \! 
< {{\mathscr M}_\bullet}/ {\mathscr N}>  \, \,
      \longrightarrow 
     K \otimes \varprojlim_i\, 
     C_{(X_\bullet, M_\bullet)/(\mathscr C_i,\mathscr N_i)}
\end{equation}
and it induces the comparison theorem in \cite{shiho:rel1} 2.36:

\begin{thr}\label{basechange3} The canonical morphism \eqref{crysconv} induces an isomorphism 
$$
\R^{m} f_{(X,M) / (\mathscr C, {\mathscr N}), {\rm conv} \ast } ({\O}_{\mathrm{conv}, X,K}) \iso
\R^{m} f_{(X,M) / (\mathscr C, {\mathscr N}), {\rm crys} \ast } ({\O}_{\mathrm{crys}, X,K})
$$
of  iso-coherent sheaves  on $\mathscr C$ such that the Frobenius structures on both sides commute. 
\end{thr}

\vspace*{3mm}

\begin{rmk}
To any locally free log-convergent isocrystal $\cal E$, 
it is possible to associate a log-crystalline isocrystal $\Phi(\cal E)$ \cite{shiho:rel1} 2.35. 
Here we should have written $\Phi ({\cal O}_{\mathrm{conv}, X,K})$
for the structural log-convergent isocrystal ${\cal O}_{\mathrm{crys}, X,K}$.
But $\Phi ({\cal O}_{\mathrm{conv}, X,K})$ is the structural log-crystalline isocrystal. 
Hence we prefer to omit $\Phi$.
\end{rmk}

As a corollary, if we specialize the above theorem to the case presented in (8), we have ([SH08] 2.38):

\begin{thr}\label{basechangeconv1} With the notation above, there is a natural isomorphism
$$ 
\mathbb L\hat{\iota}^{\ast}\R f_{(X,M)/(\mathscr C, {\mathscr N}), {\rm conv} \ast}({\O}_{\mathrm{conv}, X,K})  
\iso
{\R f_{s}}_{(X_{s}, M_s) / {{\cal V}^\times}, {\rm conv} \ast } ({\O}_{\mathrm{conv}, X_s, K})
$$
in the derived category of perfect $K$-complexes.
\end{thr}
 
Let us introduce, after \cite{shiho:rel1}, {the {\it log-analytic cohomology sheaves} 
 {\it  of}  $(X,M) /(C, N)$ with respect to $(\mathscr C, {\mathscr N})$ which is indicated and defined by 
\begin{equation*} 
	\R^{m}  f_{(X,M) / (\mathscr C, {\mathscr N}), {\rm an} \ast} (\O_{\mathrm{an}, X, K}) 
	= \R^m g_{K \ast}^{\mathrm{ex}} 
	\Omega^{\bullet}_{{] {{ X}_{\bullet}}[^{\rm log}_{{{\mathscr P}}_{\bullet}}}/{{\mathscr C}_{K}}} \! \!  \! 
< {{\mathscr M}_\bullet}/ {\mathscr N}>. 
\end{equation*}
in \cite{shiho:rel1} 4.1, 
where $g_K^{\mathrm{ex}} : ]X_\bullet[_{{\mathscr P}_\bullet}^{\mathrm{log}} \rightarrow \mathscr C_K$ 
is the induced morphism from $g : \mathscr P_\bullet \rightarrow \mathscr C$. 
Then they are $\O_{\mathscr C_K}$-sheaves.  
By \cite{shiho:rel1} 4.6 we have

\begin{thr} \label{logan} In the previous notation, there is an isomorphism 
$$
\mathrm{sp}_{\ast}\R^m  f_{(X,M) / (\mathscr C, {\mathscr N}), {\rm an} \ast } ({\O}_{{\rm an},X,K}) \iso 
\R^m f_{(X,M) / (\mathscr C, {\mathscr N}), {\rm conv} \ast } ({\O}_{{\rm conv},X,K}), 
$$
for any $m$ which is compatible with Frobenius maps, 
where $\mathrm{sp} : \mathscr C_K \rightarrow \mathscr C$ 
is the specialization morphism. Moreover, 
the log-analytic cohomology 
$\R^{m} f_{(X,M) / (\mathscr C, {\mathscr N}), {\rm an} \ast }({\O}_{\mathrm{an}, X,K})$ 
is a coherent $\O_{{\mathscr C_{K}}}$-sheaf such that the Frobenius map 
is an isomorphism. 
\end{thr}

\begin{proof} The assertion follows from \cite{shiho:rel1} 4.6 except that on the Frobenius structures on the 
log-analytic cohomology. 
By definition $\R^m f_{(X,M) / (\mathscr C, {\mathscr N}), {\rm an} \ast } ({\O}_{{\rm an}, X,K})$ 
is endowed with a Frobenius map which becomes the Frobenius structure 
on the log-convergent cohomology sheaf 
$\R^m f_{(X,M) / (\mathscr C, {\mathscr N}), {\rm conv} \ast } ({\O}_{{\rm conv},X,K})$ 
after taking the direct image by the specialization morphism $\mathrm{sp}$. 
If $\mathrm{sp}_{\ast} {\cal F}=0$ 
for a coherent sheaf $\cal F$ on ${\mathscr C}_K$, then one can argue that ${\cal F}=0$,  and so we may 
conclude that the Frobenius map is an isomorphism on 
$\R^m f_{(X,M)/(\mathscr C, {\mathscr N}), {\rm an} \ast}({\O}_{{\rm an}, X,K})$.  
\end{proof} 

\medskip

 $$ \ast  \ast \ast$$
 
So far we have not discussed the differential structure on the cohomology sheaves. 
We are going to do that in this subsection and we will refer to it as "{\it the Gauss-Manin connection}" with respect 
the composite of log-smooth morphisms 
$$
(X, M) \rightarrow (C, N) \rightarrow \mathrm{Spec}\, k. 
$$

As is the same with the case of usual crystalline cohomology \cite{Ber:Cri74} V, 3.6, 
the log-crystalline cohomology 
$\R(f\theta)_\ast C_{(X_\bullet, M_\bullet)/(\mathscr C_i, \mathscr N_i)}$ 
is a log-crystal in the derived category of $\mathcal O_{\mathrm{crys}, (C, N)/\mathcal V_i}$-modules 
by the base change theorem of log-crystalline 
cohomology for quasi-coherent and flat log-crystals \cite{Ka:log} 6.10. 
Let us put $\mathscr C(1) = (\mathscr C, \mathscr N) \times_{\mathrm{Spec}\, \mathcal V} (\mathscr C, \mathscr N)$  
(resp. $\mathscr P_\bullet(1) = 
(\mathscr P_\bullet, \mathcal M_\bullet) \times_{\mathrm{Spec}\, \mathcal V} (\mathscr P_\bullet, \mathcal M_\bullet)$), 
$\mathscr D(1)$ the $p$-adically complete log-PD-envelope of the diagonal embedding 
$(C, N) \rightarrow \mathscr C(1)$ ($\mathscr D_i(1)$ its reduction modulo $p^i$) \cite{Ka:log} 5.3, 
and $\mathrm{pr}_\alpha : \mathscr D(1) \rightarrow (\mathscr C, \mathscr N)$ the natural $\alpha$-th projection for $\alpha=1, 2$. 
If $C_{(X_\bullet, M_\bullet)/\mathscr D_i(1)}$ is the crystalline complex associated to the log-crystal 
$\mathcal O_{\mathrm{crys} X/\mathcal V}$ modulo $p^i$ on Zariski site of $X_\bullet$ which is defined by 
the logarithmic de Rham complex of the log-PD-envelope of the diagonal closed immersion 
$(X_\bullet, M_\bullet) \rightarrow \mathscr P_\bullet(1)$, 
then the crystalline cohomology of $(X, M)$ with respect to $\mathscr D_i(1)$ is given by 
$$
    \R f_{(X, M)/\mathscr D_i(1), {\rm crys} \ast} ({\O}_{{\rm crys} X/\mathcal V_i}) 
 = \R(f\theta)_\ast C_{(X_\bullet, M_\bullet)/\mathscr D_i(1)}
$$
\cite{Ka:log} 6.4. Since $\mathrm{pr}_\alpha$ is flat,  the base change theorem induces canonical isomorphisms 
$$
     \mathrm{pr}_2^\ast\R^m f_{(X, M)/(\mathscr C_i, {\mathscr N}_i), {\rm crys} \ast}({\O}_{{\rm crys}, X/\mathcal V_i})\, 
     \displaystyle{\mathop{\rightarrow}^{\cong}}\, 
     \R^mf_{(X, M)/\mathscr D_i(1), {\rm crys} \ast}({\O}_{{\rm crys} X/\mathcal V_i})
     \displaystyle{\mathop{\leftarrow}^{\cong}}\, 
     \mathrm{pr}_1^\ast\R^m f_{(X, M)/(\mathscr C_i, {\mathscr N}_i), {\rm crys} \ast } ({\O}_{{\rm crys} X/\mathcal V_i}). 
$$
The collection of these isomorphisms for all $i$ forms a HPD-stratification 
(see \cite{shiho:fund1} 4.3.1 and \cite{shiho:rel1} 1.21) and it induces the Gauss-Manin connection 
$$
   \nabla^{\mathrm{GM}}_{\mathrm{crys}} : \R^m f_{(X,M) / (\mathscr C, {\mathscr N}), {\rm crys} \ast } ({\O}_{{\rm crys},X,K})
   \rightarrow  \R^m f_{(X,M) / (\mathscr C, {\mathscr N}), {\rm crys} \ast } ({\O}_{{\rm crys}, X,K}) 
   \otimes_{\mathcal O_{\mathscr C}} \Omega_{(\mathscr C, \mathscr N)/\mathcal V}^1. 
$$
for the log-crystalline cohomology by Proposition \ref{consist} and \cite{Ka:log} 6.2. 
Moreover, the connection $\nabla^{\mathrm{GM}}_{\mathrm{crys}}$ is horizontal with respect to the Frobenius structure 
induced by our chosen Frobenius endomorphism on $(\mathscr P_\bullet, \mathcal M_\bullet)$.

In the case of log-convergent (resp. log-analytic) cohomology we have a connection 
as a following sense according to \cite{shiho:rel1} 4.10.

\begin{prp} \label{convcon} There is a unique isocrystal $\mathcal F$ 
on the log-convergent site $((C, N)/\mathcal V)_{\mathrm{conv}}^{\mathrm{log}}$ such that, if 
$\mathcal F^{\mathrm{an}} = \mathrm{sp}^\ast\mathcal F$ is the associated logarithmic connection on 
$(\mathscr C_K, \mathscr N_K)$ for the specialization morphism $\mathrm{sp} : \mathscr C_K \rightarrow \mathscr C$, 
then the stratification 
$$
       \mathrm{pr}_2^\ast \mathcal F^{\mathrm{an}}\, \displaystyle{\mathop{\longrightarrow}^{\cong}}\, 
       \mathrm{pr}_1^\ast \mathcal F^{\mathrm{an}}
$$
induced from the connection of $\mathcal F^{\mathrm{an}}$ is given by the canonical composite of base change isomorphisms 
$$
     \mathrm{pr}_2^\ast \R f_{(X, M)/(\mathscr C, \mathscr N), \mathrm{an} \ast}(\O_{\mathrm{an},X, K})\, 
     \displaystyle{\mathop{\rightarrow}^{\cong}}\, 
     \R f_{(X, M)/\mathscr C(1), \mathrm{an} \ast}(\O_{\mathrm{an},X, K})
     \displaystyle{\mathop{\leftarrow}^{\cong}}\, 
     \mathrm{pr}_1^\ast \R f_{(X, M)/(\mathscr C, \mathscr N), \mathrm{an} \ast}(\O_{\mathrm{an},X, K}). 
$$
Here $\mathrm{pr}_\alpha : ]C[_{\mathscr C(1)}^{\mathrm{log}} \rightarrow ]C[_{\mathscr C}$ 
is the $\alpha$-th projection for $\alpha=1, 2$. 
Moreover, the Frobenius structures on cohomologies above commutes with the isomorphisms. 
\end{prp}

Hence we have the Gauss-Manin connections 
$$
     \begin{array}{l}
       \nabla_{\mathrm{conv}}^{\mathrm{GM}} : 
       \R^m f_{(X, M)/(\mathscr C, \mathscr N), \mathrm{conv} \ast}(\O_{\mathrm{conv},X, K}) 
       \rightarrow \R f_{(X, M)/(\mathscr C, \mathscr N), \mathrm{conv} \ast}(\O_{\mathrm{conv},X, K}) 
       \otimes_{\mathcal O_{\mathscr C}} \Omega_{(\mathscr C, \mathscr N)/\mathcal V}^1 \\
              \nabla_{\mathrm{an}}^{\mathrm{GM}} : 
       \R f_{(X, M)/(\mathscr C, \mathscr N), \mathrm{an} \ast}(\O_{\mathrm{an},X, K}) 
       \rightarrow \R f_{(X, M)/(\mathscr C, \mathscr N), \mathrm{an} \ast}(\O_{\mathrm{an},X, K}) 
       \otimes_{\mathcal O_{]C[_{\mathscr C}}} \Omega_{]C[_{\mathscr C}/K}^1<\mathscr N>
       \end{array}
$$
on both log-convergent and log-analytic cohomologies such that the Frobenius structures are horizontal. 

Let us now compare $\nabla_{\mathrm{crys}}^{\mathrm{GM}}$ with $\nabla_{\mathrm{conv}}^{\mathrm{GM}}$. 
If we denote the log-PD-envelope of the closed immersion 
$(X_\bullet, M_\bullet) \rightarrow (\mathcal P_\bullet, \mathcal M_\bullet)$ 
(resp. $(X_\bullet, M_\bullet) \rightarrow \mathcal P_\bullet(1)$) 
by $\mathscr{DP}_\bullet$ (resp. $\mathscr{DP}_\bullet(1)$), 
then the natural diagram 
$$
\left[\begin{array}{ccccc}
           \mathscr{DP}_\bullet &\overset{\mathrm{pr_1}}\leftarrow &\mathscr{DP}_\bullet(1) 
           &\overset{\mathrm{pr_2}}\rightarrow &\mathscr{DP}_\bullet \\
           \downarrow & &\downarrow & &\downarrow \\
          (\mathscr P_\bullet, \mathscr M_\bullet) &\overset{\mathrm{pr_1}}\leftarrow &\mathscr{P}_\bullet(1) 
           &\overset{\mathrm{pr_2}}\rightarrow&(\mathscr P_\bullet, \mathscr M_\bullet) \\
           \uparrow & &\uparrow & &\uparrow \\
          ]X_\bullet[_{\mathscr P_\bullet}^{\mathrm{log}} &\overset{\mathrm{pr_1}}\leftarrow 
          &]X_\bullet[_{\mathscr P_\bullet(1)}^{\mathrm{log}} 
           &\overset{\mathrm{pr_2}}\rightarrow &]X_\bullet[_{\mathscr P_\bullet}^{\mathrm{log}}  \\
          \end{array}\right]\, \, 
          \mathbf{\longrightarrow}\, \, 
\left[\begin{array}{ccccc}
           (\mathscr C, \mathscr N) &\overset{\mathrm{pr_1}}\leftarrow &\mathscr D(1) 
           &\overset{\mathrm{pr_2}}\rightarrow &(\mathscr C, \mathscr N) \\
           \downarrow & &\downarrow & &\downarrow \\
          (\mathscr C, \mathscr N) &\overset{\mathrm{pr_1}}\leftarrow &\mathscr C(1) 
           &\overset{\mathrm{pr_2}}\rightarrow&(\mathscr C, \mathscr N) \\
           \uparrow & &\uparrow & &\uparrow \\
          ]C[_{\mathscr C} &\overset{\mathrm{pr_1}}\leftarrow 
          &]C[_{\mathscr C(1)}^{\mathrm{log}} 
           &\overset{\mathrm{pr_2}}\rightarrow &]C[_{\mathscr C}  \\
          \end{array}\right]
$$
of ringed $G$-spaces with log-structures 
is commutative (the middle arrow means a map between the corresponding entries in the two parts of the diagram).  
Moreover, if there exists a global chart, then the left part of  diagram could be replaced by 
$$
\begin{array}{ccccc}
           \mathscr{DP}_\bullet &\overset{\mathrm{pr_1}}\leftarrow &\mathscr{DP}_\bullet(1) 
           &\overset{\mathrm{pr_2}}\rightarrow &\mathscr{DP}_\bullet \\
           \downarrow & &\downarrow & &\downarrow \\
          \widehat{\mathscr P}_\bullet^{\mathrm{ex}} &\overset{\mathrm{pr_1}}\leftarrow &\widehat{\mathscr P}_\bullet^{\mathrm{ex}}(1) 
           &\overset{\mathrm{pr_2}}\rightarrow&\widehat{\mathscr P}_\bullet^{\mathrm{ex}} \\
           \uparrow & &\uparrow & &\uparrow \\
          ]X_\bullet[_{\mathscr P_\bullet}^{\mathrm{log}} &\overset{\mathrm{pr_1}}\leftarrow 
          &]X_\bullet[_{\mathscr P_\bullet(1)}^{\mathrm{log}}
           &\overset{\mathrm{pr_2}}\rightarrow &]X_\bullet[_{\mathscr P_\bullet}^{\mathrm{log}}, 
          \end{array}
$$
where $\mathscr P_\bullet^{\mathrm{ex}}$ (resp. $\mathscr P_\bullet(1)^{\mathrm{ex}}$) is the exactification of the closed immersion 
$(X_\bullet, \mathscr M_\bullet) \rightarrow (\mathscr P_\bullet, \mathscr M_\bullet)$ 
(resp. $(X_\bullet, \mathscr M_\bullet) \rightarrow \mathscr P_\bullet(1)$) and 
$\widehat{\mathscr P}_\bullet^{\mathrm{ex}}$ (resp. $\widehat{\mathscr P}_\bullet(1)^{\mathrm{ex}}$) 
is the completion of $\mathscr P_\bullet^{\mathrm{ex}}$ (resp. $\mathscr P_\bullet(1)^{\mathrm{ex}}$) 
along $X_\bullet$. 
Since there exists a chart of the log-structure $\mathscr M_\bullet$ Zariski locally on $\mathscr P_\bullet$ and 
the log-tube $]X_\bullet[_{\mathscr P_\bullet}^{\mathrm{log}} 
= ]X_\bullet[_{\mathscr P_\bullet^{\mathrm{ex}}} = \widehat{\mathscr P}_{\bullet K}^{\mathrm{ex}}$ 
(resp. $]X_\bullet[_{\mathscr P_\bullet(1)}^{\mathrm{log}} 
= ]X_\bullet[_{\mathscr P_\bullet^{\mathrm{ex}}(1)} = \widehat{\mathscr P}_\bullet^{\mathrm{ex}}(1)_K$)
is independent of the choice of charts \cite{shiho:fund2} 2.2.4, 
we have a commutative diagram 
{\small $$
    \begin{array}{ccccc}
        \mathrm{pr}_2^\ast \mathrm{sp}_\ast\R^m f_{(X, M)/(\mathscr C, \mathscr N), \mathrm{an} \ast}(\O_{\mathrm{an},X, K})
     &\displaystyle{\mathop{\rightarrow}^{\cong}}
    & \mathrm{sp}_\ast\R^m f_{(X, M)/\mathscr C(1), \mathrm{an} \ast}(\O_{\mathrm{an},X, K})
     &\displaystyle{\mathop{\leftarrow}^{\cong}}
     &\mathrm{pr}_1^\ast \mathrm{sp}_\ast\R^m f_{(X, M)/(\mathscr C, \mathscr N), \mathrm{an} \ast}(\O_{\mathrm{an},X, K}) \\
     \downarrow & &\downarrow & &\downarrow \\
         \mathrm{pr}_2^\ast\R^m f_{(X, M)/(\mathscr C, {\mathscr N}), {\rm crys} \ast}({\O}_{{\rm crys}, X, K})
     &\displaystyle{\mathop{\rightarrow}^{\cong}}
     &\R^mf_{(X, M)/\mathscr D(1), {\rm crys} \ast}({\O}_{{\rm crys}, X, K})
    &\displaystyle{\mathop{\leftarrow}^{\cong}}
     &\mathrm{pr}_1^\ast\R^m f_{(X, M)/(\mathscr C, {\mathscr N}), {\rm crys} \ast } ({\O}_{{\rm crys}, X, K}) 
     \end{array}
$$}

\noindent
of $\mathcal O_{\mathscr C(1)}$-modules, where $\mathrm{sp}$ of both sides in the upper line is 
$\mathrm{sp} : \mathscr C_K \rightarrow \mathscr C$ and that of the middle is 
$\mathrm{sp} : ]C[_{\mathscr C(1)}^{\mathrm{log}} \rightarrow \widehat{\mathscr C}(1)$ for 
the completion $\widehat{\mathscr C}(1)$ of $\mathscr C(1)$ along $C$. 
Here we use the flatness of $\mathrm{pr}_\alpha$ to take the cohomology sheaves 
of degree $m$. 
Hence we have the proposition below.

\begin{prp} Under the isomorphism between the log-crystalline cohomology and the log-convergent cohomology 
in Theorem \ref{basechange3}, two Gauss-Manin connections $\nabla_{\mathrm{crys}}^{\mathrm{GM}}$ 
and $\nabla_{\mathrm{conv}}^{\mathrm{GM}}$ 
coincide with each other. 
\end{prp}

\bigskip

By the existence of integrable connection on 
$\R^m  f_{(X,M)/(\mathscr C, {\mathscr N}), {\rm an} \ast }({\O}_{{\rm an}, X,K})$, which is 
horizontal with respect to the Frobenius structure, we have 

\begin{thr} \label{relan} With the previous notation, 
the log-analytic cohomology sheaf 
$\R^m  f_{(X,M)/(\mathscr C, {\mathscr N}), {\rm an} \ast }({\O}_{{\rm an}, X,K})$ 
is locally free on the rigid analytic space ${\mathscr C}_K$ 
furnished with the G-topology. 
\end{thr}

\begin{proof} Outside $\hat s_K$ in ${\mathscr C}_{K}$, 
the log-connection is a  usual connection, hence the sheaf is locally free, because we are in characteristic $0$. 
The problem is at $\hat  s_K$. Here it is enough  to check that there is no nontrivial torsion. 
The existence of the Frobenius structure forces an isomorphism between the original module and its transform by Frobenius. 
If we had a nontrivial torsion, we would have an isomorphism between modules with different lengths. 
This is a contradiction.
\end{proof}

\begin{crl} \label{locfree} The log-crystalline (resp. log-convergent) cohomology sheaf
$\R^m  f_{(X,M)/(\mathscr C, {\mathscr N}), {\rm crys} \ast }({\O}_{{\rm crys}, X,K})$ 
(resp. $\R^m  f_{(X,M)/(\mathscr C, {\mathscr N}), {\rm conv} \ast }({\O}_{{\rm conv}, X,K})$) 
is locally projective as $\O_{\mathscr C} \otimes K$-modules. In particular, we have a diagram 
where all the maps are  isomorphisms of finite dimensional $K$-vector spaces for all $m$:
\begin{equation} \label{diag4}
\xymatrix{
H^m_{ \rm log-conv}( (X_s,M_s)/{\cal V}^\times) \ar[d]\ar[r]
&H^{m}_{\rm log-crys}((X_s,M_s)/{\cal V}^\times) \otimes K \ar[d]\\
\R ^{m}  f_{(X,M) / (\mathscr C, {\mathscr N}), {\rm conv} \ast } ({\O}_{\mathrm{conv}, X,K})_{s^{\times}} \ar[r] 
&\R^{m} f_{(X, M)/(\mathscr C, {\mathscr N}), {\rm crys} \ast}({\O}_{\mathrm{crys}, X,K})_{s^{\times}}.}
\end{equation}
\end{crl}

As a summary of the results, so far, 
we have the following isomorphism of finite dimensional $K$-spaces with Frobenius structures 
for the pullback at ${\hat s}^{\times}_K$ 
of $\R^mf_{(X,M)/(\mathscr C, {\mathscr N}), {\rm an} \ast}({\O}_{{\rm an}, X,K})$ 
({\it i.e.}, as in \eqref{diag3},  $t \mapsto 0$)  in ${\mathscr C}_K$:
\begin{equation} \label{concl1}
\R^mf_{(X,M)/ (\mathscr C, {\mathscr N}), {\rm an} \ast} ({\O}_{{\rm an}, X,K})_{{\hat s}^{\times}_K} 
\iso H^m_{\rm log-crys}((X_s,M_s)/{\cal V}^\times) \otimes K 
\end{equation}

 \medskip
 $$ \ast  \ast \ast$$

There is another way to introduce an integrable connection 
on the log-crystalline (resp. log-convergent, 
resp. log-analytic) cohomology sheaf 
$\R^m  f_{(X,M) / (\mathscr C, {\mathscr N}), {\rm crys} \ast } ({\O}_{\mathrm{crys}, X,K})$ 
(resp. $\R^m  f_{(X,M) / (\mathscr C, {\mathscr N}), {\rm conv} \ast } ({\O}_{\mathrm{conv}, X,K})$, 
resp. $\R^m  f_{(X,M) / (\mathscr C, {\mathscr N}), {\rm an} \ast } ({\O}_{\mathrm{an}, X,K})$) 
by using the spectral sequences studied by 
Katz and Oda \cite{KO}, sect. 3. The monodromy operator on the log-crystalline 
cohomology $H^m_{ \rm log-crys}( (X_s,M_s)/{\cal V}^\times) \otimes K$ introduced 
by Hyodo and Kato \cite{hy-ka} 3.6 is along this line. 

Let us now recall the construction. 
Let $C_{(X_\bullet. M_\bullet)/\mathcal V_i}^\bullet$ 
(resp. $\Omega_{]X_\bullet[_{\mathscr P_\bullet}^{\mathrm{log}}/K}^\bullet<{\mathscr M}_\bullet>$)
be the logarithmic de Rham complex of the log-PD-envelope 
of the closed immersion 
$\iota_\bullet : (X_\bullet, M_\bullet) \rightarrow (\mathscr P_\bullet, \mathscr M_\bullet)$ 
over $\mathcal V_i$, where $\mathcal V_i = \mathcal V/p^i$ is endowed with the trivial log-structure, 
(resp. the logarithmic de Rham complex of the simplicial rigid analytic space 
$]X_\bullet[_{\mathscr P_\bullet}^{\mathrm{log}}$ over $K$). 
We define decreasing filtrations 
$\mathrm{Fil}\sp q_{\mathrm{crys}, \mathscr C_i}$ of $C_{(X_\bullet, M_\bullet)/\mathcal V_i}^\bullet$ 
and $\mathrm{Fil}\sp q_{\mathrm{an}}$ (resp. $\mathrm{Fil}\sp q_{\mathrm{conv}}$) of 
$\Omega_{]X_\bullet[_{\mathscr P_\bullet}^{\mathrm{log}}/K}^\bullet<{\mathscr M}_\bullet>$ 
(resp. $\mathrm{sp}_\ast\Omega_{]X_\bullet[_{\mathscr P_\bullet}^{\mathrm{log}}/K}^\bullet<{\mathscr M}_\bullet>$) by 
$$
    \begin{array}{l}
     \mathrm{Fil}\sp 0_{\mathrm{crys}, \mathscr C_i} = C_{(X_\bullet, M_\bullet)/\mathcal V_i}^\bullet \\
     \mathrm{Fil}\sp 1_{\mathrm{crys}, \mathscr C_i} = \mathrm{Im}
     (C_{(X_\bullet, M_\bullet)/\mathcal V_i}^{\bullet-1} \otimes 
     \Omega^1_{(\mathscr C_i, \mathcal N_i)/\mathcal V_i} \rightarrow 
     C_{(X_\bullet, M_\bullet)/\mathcal V_i}^\bullet) \hspace*{27mm} \\
     \mathrm{Fil}\sp 2_{\mathrm{crys}, \mathscr C_i} = 0
     \end{array}
 $$
and 
$$
    \begin{array}{l}
     \mathrm{Fil}\sp 0_{\mathrm{an}} = \Omega_{]X_\bullet[_{
     \mathscr P_\bullet}^{\mathrm{log}}/K}^\bullet<{\mathscr M}_\bullet>  \\
     \mathrm{Fil}\sp 1_{\mathrm{an}} = \mathrm{Im}
     (\Omega_{]X_\bullet[_{\mathscr P_\bullet}^{\mathrm{log}}/K}^{\bullet-1}<{\mathscr M}_\bullet> 
     \otimes \Omega^1_{]C[_{\mathscr C}/K}<\mathscr N> \rightarrow 
     \Omega_{]X_\bullet[_{\mathscr P_\bullet}^{\mathrm{log}}/K}^{\bullet}<{\mathscr M}_\bullet>) \\
     \mathrm{Fil}\sp 2_{\mathrm{an}} = 0
     \end{array}
 $$
(resp. $\mathrm{Fil}\sp q_{\mathrm{conv}} = \mathrm{sp}_\ast\mathrm{Fil}\sp q_{\mathrm{an}}$ 
where $\mathrm{sp} : ]X_\bullet[_{\mathscr P_\bullet} \rightarrow {\widehat{\mathscr P}}_\bullet$ is the specialization map), 
respectively. Since 
$$
    0 \rightarrow 
    g^\ast\Omega^1_{(\mathscr C_i, \mathcal N_i)/\mathcal V_i} 
       \rightarrow
       \Omega^1_{(\mathscr P_{\bullet, \mathscr V_i}, \mathcal M_{\bullet, i})/\mathcal V_i}
      \rightarrow 
      \Omega^1_{(\mathscr P_{\bullet, \mathscr V_i}, \mathcal M_{\bullet, i})/
      (\mathscr C_i, \mathcal N_i)} 
      \rightarrow 0
$$
is an exact sequence of sheaves of locally free 
$\O_{\mathscr P_\bullet \otimes \mathcal V_i}$-modules of finite type, we have 
$$
   \begin{array}{l} 
      \mathrm{gr}^0_{\mathrm{crys}, \mathscr C_i} 
      = \mathrm{Fil}^0_{\mathrm{crys}, \mathscr C_i}/\mathrm{Fil}^1_{\mathrm{crys}, \mathscr C_i}
       = C_{(X_\bullet, M_\bullet)/(\mathscr C_i, \mathcal N_i)}^\bullet \\
      \mathrm{gr}^1_{\mathrm{crys}, \mathscr C_i} 
      = \mathrm{Fil}^1_{\mathrm{crys}, \mathscr C_i}/\mathrm{Fil}^2_{\mathrm{crys}, \mathscr C_i}
       = g^\ast\Omega^1_{(\mathscr C_i, \mathcal N_i)/\mathcal V_i} 
       \otimes_{\O_{\mathscr P_\bullet \otimes \mathcal V_i}}
       C_{(X_\bullet, M_\bullet)/(\mathscr C_i, \mathcal N_i)}^\bullet[-1], 
       \end{array}
$$
where $g : \mathscr P_\bullet \rightarrow \mathscr C$ in \eqref{HKsystem}. 
Here $(C^\bullet, d^\bullet)[-1]$ means the $-1$-shift, that is, $(C^{\bullet-1}, -d^{\bullet-1})$. 
The similar hold for the filtered complexes $(\mathrm{sp}_\ast\Omega_{]X_\bullet[_{
     \mathscr P_\bullet}^{\mathrm{log}}/K}^\bullet<{\mathscr M}_\bullet>, \mathrm{Fil}\sp q_{\mathrm{conv}})$ 
and 
$(\Omega_{]X_\bullet[_{
     \mathscr P_\bullet}^{\mathrm{log}}/K}^\bullet<{\mathscr M}_\bullet>, \mathrm{Fil}\sp q_{\mathrm{an}})$. 
By this decreasing filtration we have a spectral sequence 
$$
      \underline{E}\sb 1\sp{qr} 
      = \R\sp{q + r}(f\theta)_\ast\mathrm{gr}^q_{\mathrm{crys}, \mathscr C_i}\, 
       \Rightarrow\,  
          \R\sp{q + r}(f\theta)_\ast C_{(X_\bullet, M_\bullet)/\mathcal V_i}^\bullet
$$
where
$$
   \begin{array}{l} 
    \underline{E}\sb 1\sp{0r} 
       = \R^r(f\theta)_\ast C_{(X_\bullet, M_\bullet)/(\mathscr C_i, \mathcal V_i)}^\bullet \\
           \underline{E}\sb 1\sp{1r} 
       = \R^r(f\theta)_\ast C_{(X_\bullet, M_\bullet)/(\mathscr C_i, \mathcal V_i)}^\bullet 
       \otimes_{\O_{\mathscr C_i}} 
       \Omega_{(\mathscr C_i, {\mathscr N}_i)/\mathcal V_i}^1.
    \end{array}
$$
Then the differential 
$d\sb 1\sp{0r} : \underline{E}\sb 1\sp{0r} \rightarrow \underline{E}\sb 1\sp{1r}$ 
gives an integrable $\mathcal V_i$-connection 
$$
      d\sb 1\sp{0r} : 
      \R^r(f\theta)_\ast C_{(X_\bullet, M_\bullet)/(\mathscr C_i, \mathcal V_i)}^\bullet
      \rightarrow  \R^r(f\theta)_\ast C_{(X_\bullet, M_\bullet)/(\mathscr C_i, \mathcal V_i)}^\bullet
      \otimes_{\O_{\mathscr C_i}} 
       \Omega_{(\mathscr C_i, {\mathscr N}_i)/\mathcal V_i}^1. 
$$
By taking the limit on $i$ and tensoring  with $ K$, we have a $K$-connection 
$$
    \nabla_{\mathrm{crys}}^{\mathrm{KO}} : 
      \R ^{m}  f_{(X, M)/(\mathscr C, {\mathscr N}), {\rm crys} \ast }({\O}_{\mathrm{crys}, X,K})
      \rightarrow \R ^{m}f_{(X, M)/(\mathscr C, {\mathscr N}), {\rm crys} \ast }({\O}_{\mathrm{crys}, X,K}) \otimes_{\O_{\mathscr C}}
       \Omega_{(\mathscr C, {\mathscr N})/\mathcal V}^1. 
$$
by Proposition \ref{consist}. 
We also have the integrable connections $\nabla_{\mathrm{conv}}^{\mathrm{KO}}$ and 
$\nabla_{\mathrm{an}}^{\mathrm{KO}}$ for the log-convergent cohomology and 
the log-analytic cohomology by the similar way. 
Moreover, they are compatible via the comparison isomorphisms in Theorems \ref{basechange3} and \ref{logan}.

\begin{thr} \label{GM} Two connections $\nabla_{\mathrm{crys}}^{\mathrm{GM}}$ 
and $\nabla_{\mathrm{crys}}^{\mathrm{KO}}$ introduced above on 
the log-crystalline cohomology 
coincide with each other. The same holds for the log-convergent (resp. log-analytic) cohomology. 
\end{thr}

\begin{proof} Outside $s$ in $\mathscr C$, the coincidence is proved in 
\cite{Ber:Cri74} V, 3.6.4. Hence the two connections coincide 
on $\mathscr C$ since the cohomology sheaf 
$\R^m f_{(X,M)/(\mathscr C, {\mathscr N}), {\rm crys} \ast}(\O_{\mathrm{crys}, X,K})$ 
is a coherent and locally projective $\O_{\mathscr C}$-module. In the case of the log-convergent (resp. log-analytic) cohomology, 
the coincidence follows from the log-crystalline case and the comparison theorem. 
\end{proof}

\vspace*{3mm}

\begin{rmk} One should prove the coincidence of two definitions of Gauss-Manin connections 
in the log-crystalline (resp. log-convergent, resp. log-analytic) cohomology 
for more general situations. But in this paper we deal with what we need for $p$-adic Clemens-Schmid 
exact sequence. 
\end{rmk}

 \medskip
 $$ \ast  \ast \ast$$

The locally free module $\R^m f_{(X,M)/(\mathscr C, {\mathscr N}), {\rm an} \ast}({\O}_{{\rm an}, X,K})$ 
is free on $]s[_{\mathscr C} =D(0,1^{-})$ 
(the unit open disc) since $\Gamma(]s[_{\mathscr C}, \O_{]s[_{\mathscr C}})$ 
is a Bezout ring by \cite{Laz}, Corollary of Proposition 4. 
Then there exists a finite dimensional $K$-vector space $V_m$ endowed with a nilpotent endomorphism $N_m$ 
and a Frobenius structure $F_m$ satisfying $N_mF_m = pF_mN_m$ such that 
\begin{equation}\label{residue}
(\R^mf_{(X,M)/(\mathscr C, {\mathscr N}), {\rm an} \ast}({\O}_{{\rm an}, X,K})|_{]s[_{\mathscr C}}, 
\nabla, \varphi_m)\, 
\cong (V_m \otimes \O_{]s[_{\mathscr C}}, N_m \otimes 1 + 1 \otimes d, 
F_m \otimes \sigma_{\mathscr C}) 
\end{equation}
and that the vector space $V_m$ is isomorphic to 
$\R^mf_{(X,M)/ (\mathscr C, {\mathscr N}), {\rm an} \ast} ({\O}_{{\rm an}, X,K})_{{\hat s}^{\times}_K}$. 
Here $\nabla$ (resp. $\varphi_m$) is the induced Gauss-Manin connection (resp. the induced Frobenius structure) on 
$\R^mf_{(X,M)/(\mathscr C, {\mathscr N}), {\rm an} \ast}({\O}_{{\rm an}, X,K})$. 
In fact on the open unit disk $]s[_{\mathscr C} =D(0,1^{-})$ the connection reduces to a differential system 
with log-singularities endowed with a Frobenius structure. 
Hence it satisfies the hypotheses of Christol's transfer theorem \cite{chr}, Theorem 4. 
We may then find a new basis such that the matrix which represents the log-connection 
with respect to this basis is given by the constant matrix $N_m$ and the Frobenius action is 
given by $F_m$. $(V_m, N_m, F_m)$ is called the residue of 
$\R^mf_{(X,M)/(\mathscr C, {\mathscr N}), {\rm an} \ast}({\O}_{{\rm an}, X,K})$ 
at $\hat{s}^{\times}_K$.
     
On the log-crystalline cohomology $H^m_{\rm log-crys}((X_s, M_s)/{\cal V}^\times) \otimes K$ which is the right part of \eqref{concl1}, 
Hyodo and Kato introduced a monodromy operator 
with Frobenius structures \cite{hy-ka} 3.6. We would like now to prove the following 
comparison. 

\begin{thr} \label{monodromy} The isomorphism \eqref{concl1} induces an isomorphism 
$$
      V_m \iso H^m_{\rm log-crys}((X_s, M_s)/{\cal V}^\times) \otimes K
$$
of $K$-spaces which is compatible with nilpotent endomorphisms and Frobenius structures. 
\end{thr}

\begin{proof} It is sufficient to prove the monodromy operator $N_m$ 
coincides with Hyodo-Kato's monodromy operator 
under the isomorphism \eqref{concl1}. 
We recall the construction of the monodromy operator on the log-crystalline cohomology 
$H^m_{\rm log-crys}((X_s,M_s)/{\cal V}^\times) \otimes K$ following \cite{hy-ka} 3.6. 

Shrinking $C$, we may assume that $\mathscr C \rightarrow \mathrm{Spf}\, \widehat{\mathcal V[t]}$ 
is \'etale where $t$ is a local coordinate of $\hat{s}$ over $\mathcal V$ and $\widehat{\mathcal V[t]}$ is a $p$-adic completion 
of the polynomial ring $\mathcal V[t]$. 
Let $C_{(X, M)/\mathcal V_i}$ the crystalline complex of $(X, M)$ 
over $\mathcal V_i$ on the Zariski site on $X_\bullet$. 
We then have an exact sequence for each $n$ as in \cite{hy-ka} 3.6: 

\begin{equation}\label{crys0}
0 \rightarrow C_{(X, M)/({\mathscr C}_i, \mathscr N_i)}[-1] \rightarrow C_{(X, M)/{{\cal V}_i}} 
\rightarrow C_{(X, M)/({\mathscr C}_i, \mathscr N_i)} \rightarrow 0, 
\end{equation}

\noindent
where the first map is given by the external multiplication by $\omega \mapsto dt/t\wedge \omega$. 
Note that $(C^\bullet, d^\bullet)[-1] = (C^{\bullet-1}, -d^{\bullet-1})$ by definition. 
The exact sequence 
\begin{equation}\label{crys2}
0 \rightarrow C_{(X, M)/({\mathscr C}_i, \mathscr N_i)} \otimes_{{\cal O}_{\mathscr C_i}} {\cal V}_i[-1] 
\rightarrow C_{(X, M)/{{\cal V}_i}} \otimes_{{\cal O}_{\mathscr C_i}} {\cal V}_i
\rightarrow C_{(X, M)/({\mathscr C}_i, \mathscr N_i)} 
\otimes_{{\cal O}_{\mathscr C_i}} {\cal V}_i \rightarrow 0
\end{equation}
induced by tensoring \eqref{crys0} with ${\cal O}_{\mathscr C_i} \rightarrow {\cal V}_i$, 
which is the reduction map 
modulo the ideal of definition of the finite flat subscheme $\widehat{s}$ of $\mathscr C$ 
over $\cal V$ (then $t$ maps to $0$),  
is nothing but the second exact sequence in \cite{hy-ka} 3.6 
since $\mathscr C \rightarrow \mathrm{Spf}\, \widehat{\mathcal V[t]}$ 
is \'etale. Hence the connecting homomorphisms on the  cohomology  groups 
with respect to the short exact sequence (\ref{crys2}) induce Hyodo-Kato's monodromy operator on 
$H^m_{\rm log-crys}((X_s, M_s)/{\cal V}^\times) \otimes K$ by the isomorphism 
$$
          H^m_{\rm log-crys}((X_s, M_s)/{\cal V}^\times) \otimes K 
      \iso  (\varprojlim\R^m(f_{s} \theta_{s})_{ \ast} C_{(X, M)/({\mathscr C}_i, \mathscr N_i), s^{\times}}) \otimes K,
$$
where $C_{(X, M)/({\mathscr C}_i, \mathscr N_i), s^{\times}}$ is the fiber of 
$C_{(X, M)/({\mathscr C}_i, \mathscr N_i)}$ at $s^\times$ and it is isomorphic to the crystalline complex 
of $(X_s, M_s)/s^\times$ with respect to the embedding system 
$(X_{s, \bullet}, M_{s, \bullet}) \rightarrow ({\mathscr P}_\bullet, \mathscr M_\bullet) 
\otimes_{\O_{\mathscr C}}\, {\cal V}\, \, 
(\O_{\mathscr C_\mathcal V} \rightarrow {\cal V}, \, \,  t \mapsto 0)$ 
induced from the diagram \eqref{HKsystem}. 

Since the connecting homomorphism on the cohomology sheaves with respect 
to the short exact sequence \eqref{crys0} is just the edge homomorphism of $E_1$-terms 
of the Leray spectral sequences, it coincides with the Gauss-Manin connection on 
the crystalline cohomology sheaves by Theorem \ref{GM}. 
Hence, two monodromy operators coincides with each other under the isomorphism \eqref{concl1}. 
\end{proof}

Let us now consider an exact sequence 
\begin{equation}
\label{convseq}
      0 \rightarrow \Omega^{\bullet}_{{] {{ X}_{\bullet}}[^{\rm log}_{{{\mathscr P}}_{\bullet}}}/{{\mathscr C}_K}} \! \!  \! 
< {{\mathscr M}_\bullet}/ {\mathscr N}>[-1] 
      \rightarrow   \Omega^{\bullet}_{{] {{ X}_{\bullet}}[^{\rm log}_{{{\mathscr P}}_{\bullet}}}} \! \!  \! 
< {{\mathscr M}_\bullet}>
      \rightarrow 
      \Omega^{\bullet}_{{] {{ X}_{\bullet}}[^{\rm log}_{{{\mathscr P}}_{\bullet}}}/{{\mathscr C}_K}} \! \!  \! 
< {{\mathscr M}_\bullet}/ {\mathscr N}> \rightarrow 0
\end{equation}
of $\O_{]X_\bullet[_{{\mathscr P}_\bullet/{\mathscr C}_K}}$-modules, 
where $ \Omega^{\bullet}_{{] {{ X}_{\bullet}}[^{\rm log}_{{{\mathscr P}}_{\bullet}}}} \! \!  \! 
< {{\mathscr M}_\bullet}>$ is the logarithmic de Rham complex 
of $(X, M)/K$ with respect to the embedding system (\ref{HKsystem}) 
and the first map is given by the external multiplication by $\omega \mapsto dt/t\wedge \omega$. 
Applying $\R  \Gamma(]X_{s,\bullet}[_{{{\mathscr P}}_{\bullet}},  -)$ to \eqref{convseq}, 
we obtain an exact sequence of $K$-vector spaces 
$$
\begin{array}{ll}
\cdots \rightarrow H^m_{\rm log-conv}((X_s, M_s)/{\cal V})
&\rightarrow \Gamma(]s[_{\mathscr C},  \R^{m} f_{(X,M) / (\mathscr C, {\mathscr N}), {\rm an} \ast } ({\O}_{{\rm an},X,K})) \\
&\xrightarrow{\nabla} \Gamma(]s[_{\mathscr C},  
\R^{m} f_{(X,M) / (\mathscr C, {\mathscr N}), {\rm an} \ast } ({\O}_{{\rm an},X,K}))dt/t \rightarrow \cdots 
\end{array}
$$
where $\nabla$ is the log-connection on the free $\Gamma(] s[_{\mathscr C}, \O_{] s[_{\mathscr C}})$-module 
endowed with a Frobenius structure. 
Taking the residue of 
$\Gamma(] s[_{\mathscr C},  \R^{m}f_{(X, M)/(\mathscr C, {\mathscr N}), {\rm an} \ast} ({\O}_{X,K}))$ at  ${\hat s}^{\times}_K$ 
(see \eqref{residue}), we have the following long exact sequence of finite dimensional $K$-vector spaces
\medskip
$$
\dots \rightarrow H^m_{\rm log-conv}((X_s, M_s)/{\cal V}) \rightarrow  H^m_{\rm log-crys}((X_s,M_s)/{\cal V}^\times) \otimes K \xrightarrow{N_m}  H^m_{\rm log-crys}((X_s,M_s)/{\cal V}^\times) \otimes K(-1) \rightarrow \cdots
\leqno \eqref{rel6}
$$

\noindent 
by Theorem \ref{monodromy}, where $(-1)$ means the $(-1)$-th Tate twist of Frobenius structure.

\vspace*{3mm}

\begin{rmk}
Our results in this section also hold if we begin with a locally free isocrystal $\mathcal E$ on 
the log-convergent site $((X, M)/\mathcal V)_{\mathrm{conv}}^{\mathrm{log}}$ 
instead of $\O_{\mathrm{conv}, X, K}$. Then $\O_{\mathrm{crys}, X, K}$ is replaced by the associated 
locally free isocrystal $\Phi(\cal E)$ on the 
log-crystalline site $((X, M)/\mathcal V)_{\mathrm{crys}}^{\mathrm{log}}$ 
in \cite{shiho:fund1} 5.3.1 and \cite{shiho:rel1} 2.35. 
\end{rmk}

\bigskip

\section{Building up the  sequence}

In order to calculate the rigid cohomology of $X$ (for example) 
one does not really  need to take a covering of the variety respecting the fact that  it is defined over $C$. 
But for future use we will need the existence of coverings which are compatible with the map to the curve $C$ 
and, possibly, with a smooth  compactification.   So, we are always in the diagram   \eqref{diag2}. 
Let ${\overline C}$ denote a compactification of $C$. 
By Nagata compactification of $X$  over ${\overline C}$ we may think of having 
a compactification ${\overline X}$ of $X$ over ${\overline C}$. 
We may then suppose that we have a simplicial Zariski hypercovering diagram of the type 

\begin{equation*} 
\xymatrix{
X_{s,\bullet}\ar[d] \ar[r]^{{\iota}_{\bullet}}&X_{\bullet}  \ar[d]  \ar[r] & {\overline X}_{\bullet}  \ar[d]  \ar[r]   & {\mathscr P}_{\bullet}   \\
X_s \ar[d] \ar[r]  & X \ar[d] \ar[r] &{\overline X} \ar[d]  &  \\
s \ar[r]& C \ar[r] & {\overline C}. &
}\end{equation*}

\noindent 
The simplicial map ${\overline X}_{\bullet} \rightarrow {\overline X}$ is a Zariski affine hypercovering, 
${\mathscr P}_{\bullet}$ is a simplicial formal scheme separated of finite type 
over $\cal V$ which is smooth around $X_{\bullet}$ (note that ${\mathscr P}_{\bullet}$ is not ${\mathscr P}_{\bullet}$ in sect. 3) 
and admits a lift $\sigma_{\bullet}$ of Frobenius compatible with the Frobenius $\sigma$ on $\cal V$ 
and all squares are cartesian.  
Such a setting  allows us to calculate $H^{m}_{X_s, {\rm rig}}(X)$.  In fact  we have:

\begin{equation} \label{supportrigid}
H^{m}_{X_s,{\rm  rig}}(X)= 
\R^{m}\Gamma(]{\overline X}_{\bullet}[_{{\mathscr P}_{\bullet}}, 
[j^{\dag}_{] { X}_{\bullet}[_{{\mathscr P}_{\bullet}}}\Omega^{\bullet}_{] {\overline X}_{\bullet}[_{{\mathscr P}_{\bullet}}} \rightarrow 
j^{\dag}_{] { X}_{\bullet}\setminus X_{s,\bullet}[_{{\mathscr P}_{\bullet}}}
\Omega^{\bullet}_{] {\overline X}_{\bullet}[_{{\mathscr P}_{\bullet}}} ]), 
\end {equation}

\noindent 
where $j^\dagger_{]U_\bullet[_{{\mathscr P}_{\bullet}}}{\cal F_\bullet}$ is a simplicial sheaves of overconvergent sections along 
$]{\overline X}_\bullet \setminus U_\bullet[_{\cal P_\bullet}$ 
associated to a sheaf ${\cal F}_\bullet$ of abelian groups on $]{\overline X}_\bullet[_{\cal P_\bullet}$ 
for a simplicial open subscheme $U_\bullet$ of ${\overline X}_\bullet$.  Here $j$ is seen as the map:
$j_{]U_\bullet[_{{\mathscr P}_{\bullet}}} : {]U_\bullet[_{{\mathscr P}_{\bullet}}} \rightarrow ] {\overline X}[_{{\mathscr P}_{\bullet}}$
In particular we will have  a long exact sequence 

\begin{equation*} 
\cdots \rightarrow H^{m}_{X_s, { \rm rig}}(X) \rightarrow H^{m}_{\rm rig}(X) \rightarrow 
H^{m}_{\rm  rig}(X \setminus X_s) \rightarrow \cdots
\end{equation*}

\noindent 
of finite dimensional $K$-vector spaces with Frobenius structures.

In order to construct the exact  sequence  we seek, we will need to deal with complexes which are defined 
over $]{X}_{\bullet}[_{{\mathscr P}_{\bullet}}$. We may then consider an admissible covering 
of $] {\overline  X}_{\bullet}[_{{\mathscr P}_{\bullet}}$  given by 

\begin{equation*} 
\{ ] {\overline X}_{\bullet}[_{{\mathscr P}_{\bullet}} \setminus ] {X}_{s,\bullet}[_{{\mathscr P}_{\bullet}}, ] {X}_{\bullet}[_{{\mathscr P}_{\bullet}} \}.
\end{equation*}

\noindent 
The two complexes which form the simple complex in \eqref{supportrigid} are equal on 
$ ] {\overline X}_{\bullet}[_{{\mathscr P}_{\bullet}} \setminus ] {X}_{s,\bullet}[_{{\mathscr P}_{\bullet}}$ and in the intersection  
$(]{\overline X}_{\bullet}[_{{\mathscr P}_{\bullet}} \setminus ]{X}_{s,\bullet}[_{{\mathscr P}_{\bullet}}) \cap ]{X}_{\bullet}[_{{\mathscr P}_{\bullet}}$ 
(note that $X_s \cap ({\overline X}\setminus X)= \vuoto$). We then conclude that

\begin{equation}  \label{supportrigid1} 
[ j^{\dag}_{] { X}_{\bullet}[_{{\mathscr P}_{\bullet}}}\Omega^{\bullet}_{] {\overline X}_{\bullet}[_{{\mathscr P}_{\bullet}}} 
\rightarrow j^{\dag}_{]{X}_{\bullet}\setminus X_{0,\bullet}[_{{\mathscr P}_{\bullet}}}
\Omega^{\bullet}_{] {\overline X}_{\bullet}[_{{\mathscr P}_{\bullet}}} ] \iso
[\alpha_{\bullet, \ast}( j^{\dag}_{] { X}_{\bullet}[_{{\mathscr P}_{\bullet}}}
\Omega^{\bullet}_{]{\overline X}_{\bullet}[_{{\mathscr P}_{\bullet}}})_{\mid {]{X}_{\bullet} [_{{\mathscr P}_{\bullet}}}} 
\rightarrow \alpha_{\bullet, \ast}(j^{\dag}_{]{X}_{\bullet}\setminus X_{s,\bullet}[_{{\mathscr P}_{\bullet}}}
\Omega^{\bullet}_{]{{\overline X}_{\bullet}}[_{{\mathscr P}_{\bullet}}})_{\mid  {]{X}_{\bullet} [_{{\mathscr P}_{\bullet}}}}]
\end{equation}

\noindent 
where $\alpha_{\bullet}: ]{X}_{\bullet} [_{{\mathscr P}_{\bullet}} \rightarrow ]{\overline X}_{\bullet}[_{{\mathscr P}_{\bullet}}$ and 
of course we may re-write the second part of \eqref{supportrigid1} as

\begin{equation} 
\label{supportrigid2} 
\alpha_{\bullet, \ast}[\Omega^{\bullet}_{] { X}_{\bullet}[_{{\mathscr P}_{\bullet}}}
\rightarrow
j^{\dag}_{] { X}_{\bullet}\setminus X_{s,\bullet}[_{{\mathscr P}_{\bullet}}}\Omega^{\bullet}_{] {{ X}_{\bullet}}[_{{\mathscr P}_{\bullet}}}]
\end{equation}
where now we consider $j_{]U_\bullet[_{{\mathscr P}_{\bullet}}} : 
{]U_\bullet[_{{\mathscr P}_{\bullet}}} \rightarrow ] { X}[_{{\mathscr P}_{\bullet}}$.

Hence we would like to compare $H^m_{X_s, {\rm rig}}(X)$ with  
\begin{equation}
\label{supportrigid11}
H^m_{X_s,{\rm rig}}((X, X)): =\R^m \Gamma({]{X}_{\bullet}[_{{\mathscr P}_{\bullet}}}, 
[\Omega^{\bullet}_{]{X}_{\bullet}[_{{\mathscr P}_{\bullet}}} \rightarrow
j^{\dag}_{]{X}_{\bullet}\setminus X_{s,\bullet}[_{{\mathscr P}_{\bullet}}}
\Omega^{\bullet}_{]{{X}_{\bullet}}[_{{\mathscr P}_{\bullet}}}])
\end{equation}
We will refer to  such a cohomology  $H^m_{X_s,{\rm rig}}((X, X))$  as 
 {\it rigid cohomology of the pair $(X, X)$} with support in $X_s$ along 
 the terminology of \cite{ct} and \cite{ts}. 
 Note that the usual rigid cohomology $H^m_{X_s,{\rm rig}}(X)$ of $X$ with support in $X_s$ 
 is the rigid cohomology $H^m_{X_s,{\rm rig}}((X, \overline{X}))$ of the pair $(X, \overline{X})$ with support in $X_s$. 
Then we claim:

\begin{prp} 
\label{ISO} There exists an isomorphism
\begin{equation*}
H^m_{X_s,{\rm rig}}(X) \iso H^m_{X_s,{\rm rig}}((X, X))
\end{equation*}
which is compatible with Frobenius structures.
\end{prp}
\begin{proof} We need to prove the acyclicity of  $\alpha _{\bullet}$ on the sheaves which appear in \eqref{supportrigid2}. 
This follows from the  fact that the open immersions 
$]X_{\bullet} [ _{{\mathscr P}_{\bullet}} \rightarrow ]{\overline X}_{\bullet}[_{{\mathscr P}_{\bullet}}$ 
and $]X_{\bullet} \setminus X_{s, \bullet} [ _{{\mathscr P}_{\bullet}} \rightarrow ]{\overline X}_{\bullet}[_{{\mathscr P}_{\bullet}}$ 
are affinoid morphisms. This comes, by base change, 
from the fact that an open immersion of curves is affine and $X_s$ is a divisor.
 \end{proof}

\noindent 
We would like now to understand the two simplicial complexes which appear  in the definition of  the 
rigid cohomology of $(X, X)$ with support on $X_s$,  $H^m_{X_s,{\rm rig}}((X, X))\iso H^m_{X_s,{\rm rig}}((X, X))$ 
(as in \eqref{supportrigid11}), in order to involve  it  in another long exact sequence.  As in \cite{ct}, 
the first complex  gives the rigid cohomology $H^m_{\rm rig}((X, X))$ of the pair $(X, X)$, 
and the second complex  gives 
the rigid cohomology $H^m_{\rm rig}((X \setminus X_s, X))$ of the pair $(X \setminus X_s, X)$. 

\vspace*{3mm}

\begin{rmk}
These cohomologies (without support) have been named "naive"  by Berthelot, 
"convergent" by \cite{ls} 7.4.4.
\end{rmk}

 \medskip
 $$ \ast  \ast \ast$$

What we now discuss is an interpretation of such a complex in terms of log-convergent cohomology. 
In order to achieve this goal we have to choose a different good embedding than the original ${\mathscr P}_{\bullet}$ 
used thus far: it will enjoy some exactness properties.

\begin{prp} \label{emb} It is possible to construct a simplicial \'etale hypercovering $X_\bullet$ of $X$ 
which admits a closed immersion in a simplicial smooth formal scheme 
$\mathcal Q_\bullet^{\mathrm{ex}}$ separated of finite type over ${\cal V}$.  
Moreover if we endow $X_\bullet$ with the log-structure $M_\bullet$ induced from the log-structure $M$ on $X$, 
then we can give to 
$\mathcal Q_\bullet^{\mathrm{ex}}$ a fine and saturated log-structure $\mathscr M_\bullet$, 
which comes from a log-structure on the Zariski topology of $\mathcal Q_\bullet^{\mathrm{ex}}$, 
in such a way  that 
we have a commutative diagram

\begin{equation*} 
\xymatrix{
(X_\bullet, M_\bullet) \ar[d] \ar[r]^{i^{\mathrm{ex}}_\bullet}
&(\mathcal Q_\bullet^{\mathrm{ex}}, \mathscr M_\bullet)  \ar[d]   \\
 (X,M) \ar[r] & \mathcal V}
\end{equation*}

\noindent where $i^{\mathrm{ex}}_\bullet$  is an exact closed immersion of log-schemes and 
$(\mathcal Q_\bullet^{\mathrm{ex}}, \mathscr M_\bullet)$ is formally log-smooth over $\mathcal V$ 
and that it admits a lift $\sigma_\bullet$ of Frobenius which is compatible 
with the Frobenius $\sigma$ of $\mathcal V$.
\end{prp}

For our strategy of proof we will construct an \'etale hypercovering $X_\bullet$ such that, 
for each $m$, there is a disjoint decomposition 
$X_m = \coprod_{\alpha \in I_m}X_{m, \alpha}$ 
(each component is not necessarily connected) which is compatible with the simplicial structure 
and that the induced log-structure $M_{m, \alpha}$ on $X_{m, \alpha}$ from $M$ has a chart. 
Then one can explicitly write down 
the exactness procedure, which is explained in \cite{Ka:log} 4.10 
and \cite{shiho:fund2} 2.2.1. 
In order to construct an \'etale hypercovering of $X$, we use the coskeleton functor 
and, to construct an embedding system, we use a $\Gamma$-construction, 
which were studied in \cite{ct} 11.2 and \cite{ts} 7.2, 7.3. 
The proof will end at Lemma \ref{ehc4}. 

First we recall truncated simplicial (formal) schemes and the coskeleton functors. 
Let $\Delta$ be the standard simplicial category such that 
an object of this category will be indicated by 
$[l]= \{ 0,1, \dots l \}$ for any non negative  integer $l$. Let us also denote, 
for a nonnegative integer $q$, the full subcategory of $\Delta$ whose set of objects consists 
of all $[l]$ with $l \leq q$ by $\Delta[q]$. 
A $q$-truncated simplicial (formal) scheme 
over a (formal) scheme $S$ is a contravariant functor from $\Delta[q]$ 
to the category of (formal) schemes over $S$. 
For example, a $1$-truncated simplicial (formal) scheme $Y_{\bullet \leq 1}$ over $S$ is represented by 
\begin{equation}
\label{1-skeleton}
      Y_{\bullet \leq 1} = \left[Y_1\, \displaystyle{\mathop{\mathop{\longleftarrow}^\delta_{}}^{
      \displaystyle{\mathop{\longrightarrow}^{\pi_0}}}_{\displaystyle{\mathop{\longrightarrow}_{\pi_1}}}}\, 
      Y_0\right], 
\end{equation}
where $\pi_l : Y_1 \rightarrow Y_0$ (resp. $\delta : Y_0 \rightarrow Y_1$) 
is a morphism over $S$ corresponding to the map $[0] \rightarrow [1]\, (0 \mapsto l)$ 
(resp. the unique map $[1] \rightarrow [0]$). Then 
 $\pi_0 \circ \delta = \pi_1\circ \delta = \mathrm{id}_{X_0}$ hold. 

The $q$-skeleton functor 
$$
     \mathrm{sk}_q^S ; (\mbox{simplicial (formal) schemes over}\, S)\, \longrightarrow\, 
     (\mbox{$q$-truncated (formal) simplicial schemes over}\, S)
$$
has a right adjoint $\mathrm{cosk}_q^S$ which is called the $q$-coskeleton functor. 
For a nonnegative integer $m$, 
let $\Delta[q]/[m]$ be a category such that an object is a morphism $\xi : [l_\xi] \rightarrow [m]$ of 
$\Delta$ with $l_\xi \leq q$ and a morphism $\theta : \xi \rightarrow \eta$ is a morphism of $\Delta[q]$ 
with $\xi = \eta\circ\theta$. Then $q$-truncated simplicial (formal) scheme $Y_{\bullet \leq q}$ over $S$ 
induces a contravariant functor  from 
$\Delta[q]/[m]$ to the category of (formal) $S$-schemes, $\xi \mapsto  Y_{l_{\xi}}$. Then the $m$-th stage of 
$\mathrm{cosk}_q^S(Y_{\bullet \leq q})$ is given by 
\begin{equation}\label{prod1}
     \mathrm{cosk}_q^S(Y_{\bullet \leq q})_m 
     = \displaystyle{\mathop{\mathop{\mathrm{lim}}_{\longleftarrow}}_{\xi \in \mathrm{Ob}(\Delta[q]/[m])}}\, 
     Y_{l_\xi} 
\end{equation}
where the inverse limit is taken over the diagram of $S$-schemes over $\Delta[q]/[m]$, and 
the morphism $\pi_\zeta : \mathrm{cosk}_q^S(Y_{\bullet \leq q})_m \rightarrow 
\mathrm{cosk}_q^S(Y_{\bullet \leq q})_l$ corresponding to a morphism $\zeta : [l] \rightarrow [m]$ 
is given by the transformation 
$$
        \Delta[q]/[l] \rightarrow \Delta[q]/[m]\hspace*{5mm} \xi \mapsto \zeta \circ \xi
$$
of categories. More explicitly, let us put 
$$
     \begin{array}{l}
       T_m = \prod_{\xi \in \mathrm{Ob}(\Delta[q]/[m])} Y_{l_\xi}, \\
       S_m = \prod_{\theta : \xi \rightarrow \eta \in \mathrm{Mor}(\Delta[q]/[m])} Y_{l_\xi} 
      \end{array}
$$
where the products are taken over $S$, and define morphisms 
$h_i : T_m \rightarrow S_m$ over $S$ for $i = 1, 2$ by 
$$
       \begin{array}{ll}
       h_1((x_\xi)) = (y_{\theta : \xi \rightarrow \eta}) &y_{\theta : \xi \rightarrow \eta} = (x_\xi), \\
       h_2((x_\xi)) = (y_{\theta : \xi \rightarrow \eta}) &y_{\theta : \xi \rightarrow \eta} = \pi_\theta(x_\eta) 
       \end{array}
$$
where $\pi_\theta : Y_{l_\eta} \rightarrow Y_{l_\xi}$ is the corresponding morphism of (formal) schemes 
to $\theta : \xi \rightarrow \eta$. 
Then $\mathrm{cosk}_q^S(Y_{\bullet \leq q})_m$ is isomorphic to the fiber product of the diagram 
\begin{equation} 
\label{cart}
\xymatrix{
\mathrm{cosk}_q^S(Y_{\bullet \leq q})_m \ar[d]_{h_2^\prime} \ar[r]^{\hspace*{5mm} h_1^\prime}&T_m \ar[d]^{h_2} \\
T_m \ar[r]_{h_1} &S_m.
}
\end{equation}
When $m \leq q$, $\mathrm{cosk}_q^S(Y_{\bullet \leq q})_m$ is isomorphic to $Y_m$ by the composite of $h_1^\prime$ 
and the projection $T_m \rightarrow Y_{\mathrm{id}_{[m]}}$ since $\xi = \mathrm{id}_{[m]} \circ \xi$ 
for any object $\xi$ of $\Delta[q]/[m]$. 
If $Y_{\bullet \leq q}$ is separated over $S$, then $\mathrm{cosk}_q^S(Y_{\bullet \leq q})_m$ 
is a closed (formal) subscheme of $T_m$ by $h_1^\prime$ since $h_1$ is a closed immersion. 
Moreover, if $Y_{\bullet \leq q}$ is a $q$-truncated \'etale hypercovering of $S$, 
then $Y_{\bullet \leq q}$ is closed and open in $T_m$. 

Let us consider $\mathrm{cosk}_1^S(Y_{\bullet \leq 1})$ for a $1$-truncated simplicial 
(formal) scheme as in \eqref{1-skeleton}. The set of objects of $\Delta[1]/[m]$ consists of 
$$
     \begin{array}{l}
       \xi_i : [0] \rightarrow [m]\hspace*{5mm} \xi_i(0) = i\,\, \,  (0 \leq i \leq m), \\
       \eta_{i, j} : [1] \rightarrow [m]\hspace*{5mm}\eta_{i, j}(0) = i, \eta_{i, j}(1) = j\, \, \, (0 \leq i < j \leq m), \\
       \zeta_i : [1] \rightarrow [m]\hspace*{5mm}\zeta_i(0) = \zeta_i(1) = i\,\, \,  (0 \leq i \leq m). 
      \end{array}
$$
Then $\mathrm{cosk}_1^S(Y_{\bullet \leq 1})_m$ is characterized in $T_m$ via 
the closed immersion $h_1^{\prime}$ in \eqref{cart} by the following lemma:

\begin{lmm}\label{cosk1} $x \in T_m$ belongs to $\mathrm{cosk}_1^S(Y_{\bullet \leq 1})_m$ if and only if 
$x$ simultaneously satisfies the conditions:
$$
       \begin{array}{l}
               p_{\xi_i}(x) = \pi_0\circ p_{\eta_{i, j}}(x)\hspace*{5mm} \mathrm{for}\, \, 0 \leq i < j \leq m, \\
               p_{\xi_j}(x) = \pi_1\circ p_{\eta_{i, j}}(x)\hspace*{5mm} \mathrm{for}\, \, 0 \leq i < j \leq m, \\
               p_{\zeta_i}(x) = \delta\circ p_{\xi_i}(x)\hspace*{5mm} \mathrm{for}\, \, 0 \leq i \leq m, 
       \end{array}
$$
where $p_{\xi_i} : T_m \rightarrow Y_{l_{\xi_i}}$ (resp. $p_{\eta_{i, j}} : T_m \rightarrow Y_{l_{\eta_{i, j}}}$, 
resp. $p_{\zeta_i} : T_m \rightarrow Y_{l_{\zeta_i}}$) denotes the projection. 
\end{lmm}

\begin{proof} The conditions in the assertion are the relations $\pi_{\xi_i} = \pi_0 \circ\pi_{\eta_{i, j}}, 
\pi_{\xi_j} = \pi_1 \circ\pi_{\eta_{i, j}}, \pi_{\zeta_i} = \delta\circ\pi_{\xi_i}$ which comes from 
morphisms of $\Delta[1]/[m]$, respectively. 
All other relations deduce from these relations. 
\end{proof}

\medskip

At this point  we begin a proof of Proposition \ref{emb}. 
Let us take a $1$-truncated simplicial scheme $X_{\bullet \leq 1}$ over $X$, 
$$
      X_{\bullet \leq 1} = \left[X_1\, \displaystyle{\mathop{\mathop{\longleftarrow}^\delta_{}}^{
      \displaystyle{\mathop{\longrightarrow}^{\pi_0}}}_{\displaystyle{\mathop{\longrightarrow}_{\pi_1}}}}\, 
      X_0\right], 
$$
which satisfies the following hypotheses. 
\begin{list}{}{}
\item[(i)] $X_0 = \coprod_{\alpha \in I_0}X_{0, \alpha}$ 
with $|I_0| < \infty$ is an \'etale covering of $X$ by $\pi$ 
such that $X_{0, \alpha}$ is an affine integral scheme of finite type over $k$ for any $\alpha \in I_0$. 
\item[(ii)] $X_1 = \coprod_{\beta \in I_1}X_{1, \beta}$ with $|I_1| < \infty$ is an \'etale covering of 
$X_0 \times_X X_0$ by $\pi_0 \times \pi_1$ 
such that, for any $\beta \in I_1$, $X_{1, \beta}$ is an affine scheme of finite type over $k$ and 
$\pi_l(X_{1, \beta}) \subset X_{0, \alpha_l}\, (l = 0, 1)$ for some $\alpha_l \in I_0$, 
and $\delta^{-1}(X_{1, \beta})$ coincides with one of $X_{0, \alpha}$ 
for some $\alpha \in I_0$ or is empty. Note that $\alpha_l\, (l = 0, 1)$ and $\alpha$ 
are unique, and  we allow  the $X_{1,\beta}$'s to be not connected.
\item[(iii)] The inverse image $X_{s, 0, \alpha}$ of $X_s$ in $X_{0, \alpha}$ is 
an SNCD over $k$, which is defined by the inverse image 
of $x_{0, \alpha, 1}\cdots x_{0, \alpha, r_{l, \alpha}} = 0$ of 
an \'etale morphism $X_{0, \alpha} \rightarrow \mathbb A^{n+1}_k 
= \mathrm{Spec}\, k[x_{0, \alpha, 1}, \cdots, x_{0, \alpha, n+1}]$ for some $0 \leq r_{0, \alpha} \leq n+1$ 
such that the divisor $D_{0, \alpha, j}$ defined by the inverse image of  $x_{0, \alpha, j} = 0$ in $X_{0, \alpha}$ 
is irreducible for $1 \leq j \leq r_{0, \alpha}$ and $\cap_{j=1}^{r_{0, \alpha}}D_{0, \alpha, j} \ne \emptyset$. 
\item[(iv)] For any $\alpha, \alpha' \in I_0$, $1 \leq j \leq r_{0, \alpha}$, $1 \leq j' \leq r_{0, \alpha'}$ and $\beta \in I_1$, 
if $\pi_0^{-1}(D_{0, \alpha, j}) \cap X_{1, \beta}$ 
and $\pi_1^{-1}(D_{0, \alpha', j'}) \cap X_{1, \beta}$ have a common irreducible component, 
then $\pi_0^{-1}(D_{0, \alpha, j}) \cap X_{1, \beta} = \pi_1^{-1}(D_{0, \alpha', j'}) \cap X_{1, \beta}$ 
as schemes. For $\beta \in I _1$, the set $\{ D_{1, \beta, j}\, |\, 1 \leq j \leq r_{1, \beta}\}$ 
denotes all the collection of divisors of $X_{1, \beta}$ such that each of them is of the form 
$\pi_1^{-1}(D_{0, \alpha, j}) \cap X_{1, \beta}\, (\ne \emptyset)$ for some 
$\alpha \in I_0$ and $1 \leq j \leq r_{0, \alpha}$. 
Note that $D_{1, \beta, j}$ is reduced and the inverse image $X_{s, 1, \beta}$ of $X_s$ in $X_{1, \beta}$ 
is the  union of $\{ D_{1, \beta, j}\, |\, 1 \leq j \leq r_{1, \beta}\}$. 
\end{list}
Such a $1$-truncated simplicial scheme $X_{\bullet \leq 1}$ 
is a $1$-truncated \'etale hypercovering of $X$ 
and it always exists. Indeed, one can put $X_{\bullet \leq 1}$, by definition of NCD's, as follows:
\begin{list}{}{}
\item[(a)] $X_0$ is an \'etale covering of $X$ satisfying the hypotheses (i), (iii). 
\item[(b)] $X_1$ is a disjoint sum of $X_0$ and a Zariski open covering 
$\coprod_\beta\, U_\beta \rightarrow X_0 \times_X X_0$ 
of finite type such that $U_\beta$ is an affine scheme of finite type over $k$ which satisfies 
the inclusion hypothesis for $\pi_0$ and $\pi_1$ in (ii), {\it i.e.}, $\{ U_\beta \}$ is a 
refinement of $\{X_{0, \alpha} \times_X X_{0, \alpha'}\, |\, \alpha, \alpha' \in I_0\}$. 
\item[(c)] The inverse image $U_{s, \beta}$ of $X_s$ in $U_\beta$ is an SNCD 
such that the intersection of all irreducible components of $U_{s, \beta}$ is nonempty.  
\item[(d)]  For $l = 0, 1$, 
$\pi_l : X_1 \rightarrow X_0$ is given by the identity $\mathrm{id}_{X_0}$ on $X_0$ and 
the composition of $U_\beta \rightarrow X_0 \times_X X_0$ 
with  the natural $l$-th projection $X_0 \times_X X_0 \rightarrow X_0$. 
\item[(e)] $\delta : X_0\rightarrow X_1$ is the identity onto the component $X_0$ of $X_1$.
\end{list}

\noindent
Note that the condition (c) holds if one takes a sufficiently fine Zariski covering $\{ U_\beta \}$ of $X_0 \times_X X_0$. 
By the condition (c) the inverse image of an irreducible component of the SNCD $X_{s, 0, \alpha}$ of $X_{0, \alpha}$ by 
$\pi_i\, (i = 0, 1)$ with $\pi_i(U_\beta) \subset X_{0, \alpha}$ 
is either an irreducible component of the SNCD $U_{s, \beta}$ of $U_\beta$ or empty since $\pi_i$ is etale. 
Hence the hypothesis (iv) holds. 

Let us define a simplicial scheme $X_\bullet$ over $X$ by 
$$
        X_\bullet = \mathrm{cosk}_1^X(X_{\bullet \leq 1})
$$
and a log-structure $M_\bullet$ on $X_\bullet$ by the inverse image of the log-structure $M$ on $X$. 
Then $X_\bullet$ is an \'etale hypercovering of $X$. 

We will introduce a disjoint decomposition (it means a partition) on $X_m$ for each nonnegative integer $m$. 
The disjoint decomposition of $X_0 = \coprod_{\alpha \in I_0}X_{0, \alpha}$ 
and $X_1= \coprod_{\alpha \in I_1}X_{1, \alpha}$ induce a disjoint decomposition on the fiber product 
$T_m = \prod_{\xi \in \mathrm{Ob}(\Delta[1]/[m])}X_{l_\xi}$ as 
$$
        T_m = \coprod_{A = ((\alpha_{\xi_i}), (\beta_{\eta_{i, j}}), (\gamma_{\zeta_i}))} 
        T_{m, A}, \hspace*{5mm} T_{m, A} = \prod_{0 \leq i \leq m}X_{0, \alpha_{\xi_i}} \times 
        \prod_{0 \leq i<j \leq m}X_{1, \beta_{\eta_{i, j}}} \times\prod_{0 \leq i \leq m}X_{1, \gamma_{\zeta_i}}, 
$$
where $A = ((\alpha_{\xi_i}), (\beta_{\eta_{i, j}}), (\gamma_{\zeta_i}))$ runs over 
$I_0^{m+1} \times I_1^{m(m+1)/2} \times I_1^{m+1}$, 
and all fiber products are taken over $X$. Remember that $X_m$ is a closed subscheme of $T_m$ via the closed and open immersion  
$h_1^{\prime}$ of the diagram \eqref{cart}. By Lemma \ref{cosk1} we have:

\begin{lmm} Suppose furthermore that $X$ is connected. 
Then $X_m \cap T_{m, A}$ is nonempty if and only if 
the following conditions on $A = ((\alpha_{\xi_i}), (\beta_{\eta_{i, j}}), (\gamma_{\zeta_i}))$ hold simultaneously:
$$
      \left\{\begin{array}{l}
          \pi_0(X_{1, \beta_{\eta_{i, j}}}) \subset X_{0, \alpha_{\xi_i}}\, \, \mathrm{for}\, \, 0 \leq i < j \leq m \\
          \pi_1(X_{1, \beta_{\eta_{i, j}}}) \subset X_{0, \alpha_{\xi_j}}\, \, \mathrm{for}\, \, 0 \leq i < j \leq m \\
          \delta^{-1}(X_{1, \gamma_{\zeta_i}}) = X_{0, \alpha_{\xi_i}}\, \, \mathrm{for}\, \, 0 \leq i \leq m.
      \end{array} \right.
$$
In particular we have the disjoint union 
$$
X_m= \coprod_{A}^{.}  \,( X_m \cap T_{m, A}), 
$$
where $A = ((\alpha_{\xi_i}), (\beta_{\eta_{i, j}}), (\gamma_{\zeta_i}))$ 
runs over the set of indices satisfying the conditions above. 
\end{lmm}

For a general $X$, the lemma above holds on each connected component. 
We define a disjoint decomposition on $X_m$, 
$$
     X_m = \coprod_{\alpha \in I_m}\, X_{m, \alpha}, 
$$
by the induced disjoint decomposition from that of $T_m$. 
Note that $X_{m, \alpha}$ is not connected in general. 
The new disjoint decompositions 
of $X_0$ and $X_1$ are the same with the original ones via the natural projection 
$T_m \rightarrow X_{\mathrm{id}_{[m]}} = X_m\, \, (m = 0, 1)$.

\begin{lmm}\label{ehc1} 
\begin{list}{}{}
\item[{\rm (1)}] Let $\xi :[ l ]\rightarrow [m]$ be a morphism of $\Delta$. 
For $\beta \in I_m$, there is a unique element $\alpha(\xi, \beta) \in I_l$ 
such that $\pi_\xi(X_{m, \beta}) \subset X_{l, \alpha(\xi, \beta)}$. 
\item[{\rm (2)}] Let $\xi : [l_\xi] \rightarrow [m], \eta : [l_\eta] \rightarrow [m]$ be objects of $\Delta[1]/[m]$. 
For $1 \leq j \leq r_{l_\xi, \alpha(\xi, \beta)}$ and $1 \leq j' \leq r_{l_\eta, \alpha(\eta, \beta)}$, 
if $\pi_\xi^{-1}(D_{l_\xi, \alpha(\xi, \beta), j}) \cap X_{m, \beta}$ 
and $\pi_\eta^{-1}(D_{l_\eta, \alpha(\eta, \beta), j'}) \cap X_{m, \beta}$ have a common irreducible component, 
then $\pi_\xi^{-1}(D_{l_\xi, \alpha(\xi, \beta), j}) \cap X_{m, \beta} 
= \pi_\eta^{-1}(D_{l_\eta, \alpha(\eta, \beta), j'}) \cap X_{m, \beta}$ 
as schemes. 
\end{list}
\end{lmm}

\noindent
{\it Proof.}  (1) The assertion follows from our hypotheses: in particular  (ii). 

(2) It is sufficient to prove the assertion in the case where $l_\xi = l_\eta = 0$ since 
all divisors in the level of $X_1$ comes from those of $X_0$ by  hypothesis  (iv). Suppose that 
$l_\xi = l_\eta = 0$ and $\xi(0) < \eta(0)$. 
There exists a morphism $\rho : [1] \rightarrow [m]$ such that $\pi_0\circ \pi_\rho = \pi_\xi$ 
and $\pi_1\circ \pi_\rho = \pi_\eta$. Then the assertion follows from  hypothesis (iv). 
\hspace*{\fill} $\Box$

We denote the collection $\{ D_{m, \beta, j}\, |\, 1 \leq j \leq r_{m, \beta}\}$ of reduced divisors 
of $X_{m, \beta}$ such that each of them is a form $\pi_\xi^{-1}(D_{l_\xi, \alpha(\xi, \beta), j'}) \cap X_{m, \beta}\, (\ne \emptyset)$ 
for some $\xi \in \Delta[1]/[m]$ and $1 \leq j' \leq r_{l_\xi, \alpha}$. 
Then the inverse image $X_{s, m, \beta}$ 
of $X_s$ in $X_{m, \beta}$ is a union of  divisors $\{ D_{m, \beta, j}\, |\, 1 \leq j \leq r_{m, \beta}\}$.

Let us now construct an embedding system $i_\bullet : X_\bullet \rightarrow {\cal Q}_\bullet$. 
For $l = 0, 1$, we fix an affine smooth formal scheme $\mathcal R_l = \coprod_{\alpha \in I_l}\mathcal R_{l, \alpha}$ 
separated of finite type over ${\cal V}$ with an SNCD
$\mathcal E_l =\coprod_{\alpha \in I_l}\mathcal E_{l, \alpha}$ relatively to ${\cal V}$, 
which fits into the commutative diagram over ${\cal V}$ for each $\alpha \in I_l$:
\begin{equation*} 
\xymatrix{
X_{l, \alpha} \ar[d] \ar[r]^{i_{l, \alpha}}&\mathcal R_{l, \alpha} \ar[d] \\
\vspace*{10mm} \ar[d]
&\widehat{\mathbb A}_{{\cal V}}^{s_{l, \alpha}} \ar[d] 
&\hspace*{-10mm} = \mathrm{Spf}\, {\cal V}[y_{l, \alpha, 1}, \cdots, y_{l, \alpha, s_{l, \alpha}}]\widehat{\, } \\
\mathrm{Spec}\, k[x_{l, \alpha, 1}, \cdots, x_{l, \alpha, n+1}] \ar[r] 
&\widehat{\mathbb A}_{{\cal V}}^{r_{l, \alpha}} 
&\hspace*{-10mm} = \mathrm{Spf}\, {\cal V}[y_{l, \alpha, 1}, \cdots, y_{l, \alpha, r_{l, \alpha}}]\widehat{\, }} 
\end{equation*}
($\widehat{R}$ denotes a $p$-adic completion of $R$) satisfying the following hypotheses. 
\begin{list}{}{}
\item[(I)] $i_{l, \alpha} : X_{l, \alpha} \rightarrow \mathcal R_{l, \alpha}$ is a closed immersion. 
\item[(II)] The left vertical arrow is the map coming from the hypothesis (iii) for $l=0$. In case $l=1$ 
the set $\{ x_{1, \beta, 1}, \cdots, x_{1, \beta, n+1}\}$ is a system of coordinates of $X_{1, \beta}$ 
such that $D_{1, \beta, j}$ is defined by $x_{1, \beta, j} = 0$ for $1 \leq j \leq r_{1, \beta}$.  
One can take such a system by the hypotheses (ii), (iii) and (iv). 
\item[(III)] The first right vertical arrow is smooth and the inverse 
image of $y_{l, \alpha, 1}\cdots y_{l, \alpha, s_{l, \alpha}} = 0$ 
is the SNCD $\mathcal E_{l, \alpha} = \cup_{j=1}^{s_{l, \alpha}} \mathcal E_{l, \alpha, j}$, 
and the second right vertical arrow is the natural projection 
for  $r_{l, \alpha} \leq s_{l, \alpha}$. 
\item[(IV)] The bottom arrow is defined by $y_{l, \alpha, j} \mapsto x_{l, \alpha, j}$ for $1 \leq j \leq r_{l, \alpha}$.
\item[(V)] The  inverse image $i_{l, \alpha}^{-1}\mathcal E_{l, \alpha, j}$ for $1\leq j \leq s_{l,\alpha}$ is a sum of 
$D_{l, \alpha, 1}, \cdots, D_{l, \alpha, r_{l, \alpha}}$ with multiplicities.

\item[(VI)] There exists a Frobenius $\sigma_{\mathcal R_{l, \alpha}}$ on $\mathcal R_{l, \alpha}$ 
which is an extension of the Frobenius $\sigma$ on ${\cal V}$ such that 
$\sigma_{\mathcal R_{l, \alpha}}(y_{l, \alpha, j}) = y_{l, \alpha, j}^p$ for any 
$1 \leq j \leq s_{l, \alpha}$. 
\end{list}
Indeed, such formal schemes $\mathcal R_l\, (l = 0, 1)$ exist by our  hypotheses  on $X_0$ and $X_1$.  

\vspace*{3mm}

\begin{rmk}    Strictly speaking the hypotheses above are more general than those needed for the constructions 
of the present article: in fact we will be able to have $s_{l, \alpha}=r_{l,\alpha}$ and hypothesis (V) will be automatically satisfied.  
But we think that such hypotheses  will be the correct ones 
if one seeks a functorial behavior for the results connected with our constructions.
\end{rmk}

We define a log-structure of $\mathcal R_l$ which is induced by the SNCD $\mathcal E_l$. 
This log-structure has a local chart $L_{l, \alpha} = \mathbb N^{s_{l, \alpha}}$ defined by 
$1_{l, \alpha, j} \mapsto y_{l, \alpha, j}$, where $1_{l, \alpha, j}$ is the $1$ of the $j$-th component of $L_{l, \alpha}$. 
Then $(\mathcal R_l, L_l^a) = \coprod_{\alpha \in I_l}(\mathcal R_{l, \alpha}, L_{l, \alpha}^a)$ 
is log-smooth over $\mathcal V$ and the underlying morphism of  formal schemes induces 
a closed immersion $(X_l, M_l) \rightarrow (R_l,  L_l^a)$ of log-schemes. 
Here $L_{l, \alpha}^a$ means the log-structure associated to 
the pre-log-structure $L_{l, \alpha}$. 
Moreover, $(\mathcal R_l, L_l^a)$ admits a lift of Frobenius 
$\sigma_{\mathcal R_l} = \coprod_{\alpha \in I_l}\sigma_{\mathcal R_{l, \alpha}}$. 

We define a simplicial formal log-scheme ${\cal Q}_\bullet$ by
\begin{equation}\label{prod2}
     ({\cal Q}_\bullet, \mathscr L_\bullet) = \Gamma_0^{{\cal V}}((\mathcal R_0, L_0^a))^{\leq 0} \times_{\mathrm{Spf}\, {\cal V}} 
     \Gamma_1^{{\cal V}}((\mathcal R_1, L_1^a))^{\leq 1}
     = \left(\prod_{\xi \in \mathrm{Ob}(\Delta[1]/[\bullet])} (\mathcal R_{l_\xi}, L_{l_\xi}^a)\right)_\bullet
\end{equation}
as formal log-schemes over ${\cal V}$ (see the definition and properties of 
the $\Gamma$-construction in \cite{ct} 11.2 and \cite{ts} 7.3.). Then we have a closed immersion
\begin{equation*}
       i_\bullet : (X_\bullet, M_\bullet) \rightarrow ({\cal Q}_\bullet, \mathscr L_\bullet)
\end{equation*}
of formal log-schemes over ${\cal V}$ by \cite{ct}, 11.2.4 and  11.2.7. 
Since the index sets of the products are same in \eqref{prod1} and \eqref{prod2}, 
the fiber product decomposition will induce a decomposition 
$$
    i_m = \coprod_{\alpha \in I_m}i_{m, \alpha} : \coprod_{\alpha \in I_m}(X_{m, \alpha}, M_{m, \alpha}) \rightarrow 
     \coprod_{\alpha \in I_m}({\cal Q}_{m, \alpha}, \mathscr L_{m, \alpha})
$$
of closed immersions for each $m$ and they form a map of  simplicial formal log-schemes. 
(Note that the components of $({\cal Q}_m, \mathscr L_m)$ 
which have no images from $(X_m, M_m)$ can be omitted.) 
By the product construction in \eqref{prod2}, 
$({\cal Q}_\bullet, \mathscr L_\bullet)$ 
is log-smooth over ${\cal V}$ and the underlying simplicial formal scheme ${\cal Q}_\bullet$ is smooth. 
More precisely, each log-structure $\mathscr L_{m, \alpha}$ of ${\cal Q}_{m, \alpha}$ has a chart $L_{m, \alpha}$ 
which is a product of some of $L_{l, \beta}\, (l=0, 1\, \mbox{and}\,  \, \beta \in I_l)$ and is 
isomorphic to $\mathbb N^{s_{m, \alpha}}$ with some $s_{m, \alpha} \geq r_{m, \alpha}$
such that, by reordering the  generators of $L_{m, \alpha}$, 
$1_{m, \alpha, j}$ goes to a generator of $D_{m, \alpha, j}$ in $X_{m, \alpha}$ for $1 \leq j \leq r_{m, \alpha}$ 
and $1_{m, \alpha, j}$ defines a sum of $D_{m, \alpha, 1}, \cdots, D_{m, \alpha, r_{m, \alpha}}$ 
with multiplicities for $r_{m, \alpha} < j \leq s_{m, \alpha}$ in $X_{m, \alpha}$ 
by Lemma \ref{ehc1}. 
In addition $\{ L_m = \prod_{\alpha \in I_m}L_{m, \alpha}\, |\, m \geq 0 \}$ forms a co-simplicial monoid $L_\bullet$ 
by our construction of $({\cal Q}_\bullet, \mathscr L_\bullet)$. 
The Frobenius endomorphisms $\sigma_{\mathcal R_0}$ and $\sigma_{\mathcal R_1}$ induce 
a Frobenius $\sigma_\bullet$ on 
$({\cal Q}_\bullet, \mathscr L_\bullet)$ such that it acts by the multiplication by $p$ on $L_\bullet$. 

Let us define a monoid $L_{m, \alpha}^{\mathrm{ex}}$ in the associated group $L_{m, \alpha}^{\mathrm{gr}}$ by 
$$
    L_{m, \alpha}^{\mathrm{ex}} = L_{m, \alpha}\left[ \left. 
    \pm(1_{m, \alpha, j} - \sum_{j'=1}^{r_{m, \alpha}}\gamma_{j, j'}1_{m, \alpha, j'}) \right|\, 
    \begin{array}{l}
    \mbox{ $r_{m,\alpha} < j \leq s_{m,\alpha}$, the divisor defined by the image  } \\
    \mbox{ of $y_j$  in $X_{m, \alpha}$ is $\sum_{j'=1}^{r_{m, \alpha}}\gamma_{j, j'}D_{m, \alpha, j^{\prime}}$.} 
    \end{array}
    \right]. 
$$

\begin{lmm}\label{ehc3} 
\begin{list}{}{}
\item[\mbox{\rm (1)}] The composite map 
$L_{m, \alpha} \rightarrow \Gamma({\cal Q}_{m, \alpha}, \O_{{\cal Q}_m, \alpha}) 
\rightarrow \Gamma(X_{m, \alpha}, \O_{X_m, \alpha})$ 
of monoids uniquely factors as 
$$
     L_{m, \alpha} \rightarrow L_{m, \alpha}^{\mathrm{ex}} \rightarrow \Gamma(X_{m, \alpha}, \O_{X_m, \alpha}). 
$$
\item[\mbox{\rm (2)}] The log-structure on $X_{m,\alpha}$ associated to 
the pre-log-structure  $L_{m, \alpha}^{\mathrm{ex}} 
\rightarrow \Gamma(X_{m, \alpha}, \O_{X_m, \alpha})$ is isomorphic to $M_{m, \alpha}$. 
\item[\mbox{\rm (3)}] The collection 
$\{ L_m^{\mathrm{ex}} = \prod_{\alpha \in I_m}L_{m, \alpha}^{\mathrm{ex}}\, |\, m \geq 0 \}$ 
forms a co-simplicial monoid $L_\bullet^{\mathrm{ex}}$ and 
the collection $\{ L_{m, \alpha}^{\mathrm{ex}} \rightarrow M_{m, \alpha} \}$ of homomorphisms 
of monoids induces a morphism $L_\bullet^{\mathrm{ex}} \rightarrow M_\bullet$ of co-simplicial monoids. 
\end{list}
\end{lmm}

\begin{proof} (1) If a divisor defined by the image of $y_j$ in $X_{m, \alpha}$ 
is $\sum_{j'=1}^{r_{m, \alpha}}\gamma_{j, j'}D_{m, \alpha, j'}$, 
then there is a unique unit $u$ of $\Gamma(X_{m, \alpha}, \O_{X_{m, \alpha}})$ such that the images of $y_j$ 
and $u\prod_{j'=1}^{r_{m, \alpha}}y_{j'}^{\gamma_{j, j'}}$ coincide with each other in 
$\Gamma(X_{m, \alpha}, \O_{X_{m, \alpha}})$ by Lemma \ref{ehc1} (2). 
Hence $\pm(1_{m, \alpha, j} - \sum_{j'=1}^{r_{m, \alpha}}\gamma_{j'}1_{m, \alpha, j'})$ should go to $u^{\pm 1}$. 
Since the unit $u$ is unique, we have a desired factorization. 

(2) Since $X_{s, m, \alpha} = \cup_{j=1}^{r_{m, \alpha}}D_{m, \alpha, j}$ is an SNCD of $X_{m, \alpha}$, 
the natural morphism $L_{m, \alpha} \rightarrow M_{m, \alpha}/\O_{X_{m, \alpha}}^\times$ is surjective 
at each stalk. Since we take the greatest quotient by the monomial relations of elements of $L_{m, \alpha}$ 
inside $\O_{X_{m, \alpha}}$, the homomorphism $L_{m, \alpha}^{\mathrm{ex}} \rightarrow M_{m, \alpha}$ 
is injective. 

(3) The assertion follows from our construction and Lemma \ref{ehc1}. 
\end{proof}

Let us define a simplicial log-scheme $({\cal Q}_\bullet^{\mathrm{ex}}, \mathscr M_\bullet)$ by 
$$
            {\cal Q}_\bullet^{\mathrm{ex}} = {\cal Q}_\bullet 
            \times_{\mathrm{Spf}\, \mathbb Z[L_\bullet]\widehat{\, }} \mathrm{Spf}\, \mathbb Z[L_\bullet^{\mathrm{ex}}]\widehat{\, }
$$
and the log-structure ${\mathscr M}_\bullet$ associated to the natural pre-log-structure $L_\bullet^{\mathrm{ex}} \rightarrow 
\Gamma({\cal Q}_\bullet^{\mathrm{ex}} , \O_{{\cal Q}_\bullet^{\mathrm{ex}}})$. 

\begin{lmm}\label{ehc4} The closed immersion 
$i_\bullet : (X_\bullet, M_\bullet) \rightarrow ({\cal Q}_\bullet, \mathscr L_\bullet)$
factors as 
$$
     (X_\bullet, M_\bullet)\, \displaystyle{\mathop{\longrightarrow}^{i^{\rm ex}_\bullet}}\, 
     ({\cal Q}_\bullet^{\mathrm{ex}}, \mathscr M_\bullet)\,  
     \displaystyle{\mathop{\longrightarrow}^{h_\bullet}}\,  
     ({\cal Q}_\bullet, \mathscr L_\bullet) 
$$
simplicial fine and saturated formal log-schemes over $\mathcal V$ such that 
$i^{\rm ex}_\bullet$ is an exact closed immersion and $h_\bullet$ is log-\'etale. 
Moreover, the underlying formal scheme ${\cal Q}_m^{\mathrm{ex}}$ is smooth 
over ${\cal V}$ for any $m$. 
\end{lmm}

\begin{proof} $i^{\rm ex}_\bullet$ is an exact closed immersion by Lemma \ref{ehc3} (2). 
Since $(L_\bullet^{\mathrm{ex}})^{\mathrm{gr}} = L_\bullet^{\mathrm{gr}}$, $h_m$ is log-\'etale. 
By the construction we have
$$
     \Gamma({\cal Q}_{m, \alpha}^{\mathrm{ex}}, \O_{{\cal Q}_{m, \alpha}^{\mathrm{ex}}}) 
     =  \Gamma({\cal Q}_{m, \alpha}, \O_{{\cal Q}_{m, \alpha}})
     \left[\left. (y_j/\prod_{j'=1}^{r_{m, \alpha}}y_{j'}^{\gamma_{j'}})^{\pm 1}\, \right|\, 
     r_{m, \alpha} < j \leq s_{m, \alpha} \right], 
$$
where $\gamma_j$'s are as in the proof of Lemma \ref{ehc3}. 
Hence ${\cal Q}_{m, \alpha}^{\mathrm{ex}}$ is smooth over ${\cal V}$. 
\end{proof}

Finally, the Frobenius endomorphism $\sigma_{{\cal Q}_\bullet}$ on $({\cal Q}_\bullet, \mathscr L_\bullet)$ 
can be extended to the Frobenius $\sigma_\bullet$ on 
$({\cal Q}_m^{\mathrm{ex}}, \mathscr M_\bullet)$ in such a way  $\sigma_\bullet$ acts  as 
multiplication by $p$ on $L_\bullet^{\mathrm{ex}}$. 

This completes our proof of Proposition \ref{emb}. 
\hspace*{\fill} $\Box$

\vspace*{3mm}

\begin{rmk}\label{ehc0} {\rm 
\begin{list}{}{}
\item[(1)]  We could have performed a similar construction for  an \'etale hypercovering 
coming from a truncated one of any level (not only of level 1). 
For a general \'etale hypercovering we
are not able to find a similar result: we can find the good embedding
system as before only if we have a truncated system.
We observe that an alternative methods for the proofs of the present article 
would be to use   truncated systems via  the limit arguments in \cite{ts} 7.5. 
\item[(2)] 
We say  that the NCD  $X_s$ of $X$ has {\it self-intersections} if 
there are some $\alpha \in I_0$ and $j \ne j'$ such that 
the images of  the generic points of $D_{0, \alpha, j}$ and 
$D_{0, \alpha, j'}$ by the \'etale covering $\pi : X_0 \rightarrow X$ (as in the hypotheses (i), (iii)) are  the same. 
If the NCD $X_s$ of $X$ does not have self-intersections, 
then one can take an \'etale covering $X_0$ of $X$ satisfying the hypotheses (i), (iii), 
and put $X_1 = X_0 \times_X X_0$, which is given by 
$X_{0,\alpha} \times_X X_{0,\alpha'}\, (\alpha, \alpha' \in I_0)$, 
with natural projections $\pi_0$, $\pi_1$ 
and the diagonal morphism $\delta$. 
Then this $1$-truncated simplicial scheme over $X$ also satisfies the hypotheses  (ii), (iv). 
In this case our $(X_\bullet, M_\bullet)$ is a \v{C}ech hypercovering of $(X_0, M_0)$ over $(X, M)$. 
Hence, if one takes a smooth lift $({\cal Q}_0, L_0^a)$ of $(X_0, M_0)$, 
then from the \v{C}ech diagram $({\cal Q}_\bullet, \mathscr L_\bullet)$ of $({\cal Q}_0, L_0^a)$ over ${\cal V}$,
one can obtain a similar exactification $({\cal Q}_\bullet^{\mathrm{ex}}, \mathscr M_\bullet)$.
 \end{list}}
\end{rmk}

\medskip

We fix a good embedding system 
$i^{\rm ex}_\bullet : (X_\bullet, M_\bullet) \rightarrow 
({\cal Q}^{\rm ex}_{\bullet}, {\mathscr M}_{\bullet})$ as in Proposition \ref{emb}. 
Then using \'etale descent  for rigid cohomology  \cite{ct} 9.1.1 and forgetting the log-structure, 
we may  write ($X_{s,\bullet}$ is the induced \'etale hypercovering of $X_s$)

\begin{equation}
\label{supportrigid5}
H^m_{X_s,{\rm rig}}(X) \iso
\R ^m \Gamma({]{X}_{\bullet}[_{{\cal Q}^{\rm ex}_{\bullet}}}, [\Omega^{\bullet}_{]{X}_{\bullet}[_{{\cal Q}^{\rm ex}_{\bullet}}}
\rightarrow
j^{\dag}_{]{X}_{\bullet}\setminus X_{s,\bullet}[_{{\cal Q}^{\rm ex}_{\bullet}}}
\Omega^{\bullet}_{]{{X}_{\bullet}}[_{{\cal Q}^{\rm ex}_{\bullet}}}]).
\end{equation}

\noindent
We are now ready to interpret the two complexes which appear in the simple complexes of the right hand side of 
\eqref{supportrigid5}. The first one is just calculating the rigid cohomology 
$H^m_{\mathrm{rig}}((X, X))$ of the pair $(X, X)$. 
The second complex is nothing but the rigid cohomology $H^m_{\mathrm{rig}}((X \setminus X_s, X))$ 
of the pair $(X \setminus X_s, X)$. 
Then we may apply Shiho's result \cite{shiho:fund2} 2.4.4:  
it says that, for the smooth log-scheme $(X_m, M_m)$ which has a Zariski type log-structure, 
its log-convergent cohomology over $\mathcal V$ 
coincides with the rigid cohomology of the pair  $(X_m \setminus X_{s, m}, X_m)$. 
Note again that $M_m$ is the log-structure of $X_m$ induced from the SNCD $X_{s, m}$ and 
the trivial log locus of $(X_m, M_m)$ is $X_m \setminus X_{s, m}$. Hence we have 
$$
      H^m_{\mathrm{rig}}((X \setminus X_s, X)) \iso H^m_{\mathrm{log-conv}}((X, M)/\mathcal V). 
$$
To avoid confusion we stress the  fact that   $H^{m}_{\rm  log-conv}((X, M)/\mathcal V)$ 
denotes  the $m$-th log-convergent cohomology groups of $(X, M)$ relative  to $\mathcal V$. 

To connect our construction to the log theory we will use  $]-[^{\mathrm{log}}$ 
to refer to log-tubes for exact or non exact closed immersions, which has already appeared in section 3: 
for  example  the log-tubes of ${X}_{\bullet} $ for a generic embedding (not exact) 
${\cal Q}_{\bullet}$ will be indicated by $]{X}_{\bullet}[^{\rm log}_{{\cal Q}_{\bullet}}$. 
Of course for our ${\cal Q}^{\rm ex}_{\bullet}$ in the proof of Proposition \ref{emb} we have  
$$
]X_{\bullet}[^{\rm log}_{{\cal Q}_{\bullet}}= 
]X_{\bullet}[_{{\cal Q}^{\rm ex}_{\bullet}}, 
$$
where the second is the "classical" tube, since the closed immersion 
$i^{\rm ex}_\bullet : (X_\bullet, M_\bullet) \rightarrow ({\cal Q}^{\rm ex}_{\bullet}, \mathscr M_\bullet)$ is 
an exactification of $(X_\bullet, M_\bullet) \rightarrow ({\cal Q}_{\bullet} , \mathscr L_\bullet)$. 
Hence the log-convergent cohomology 
$H^m_{\rm log-conv}((X, M)/{\cal V}) = H^m_{\rm log-conv}((X_\bullet, M_\bullet)/{\cal V})$ can be calculated by 
$$
\R^m\Gamma(]{{X}_{\bullet}}[^{\rm log}_{{\cal Q}^{\rm ex}_{\bullet}}, 
\Omega^{\bullet}_{]{{X}_{\bullet}}[^{\rm log}_{{\cal Q}^{\rm ex}_{\bullet}}}\! \! < \! {\mathscr M}_{\bullet} \! >) =
\R^m\Gamma(]{{X}_{\bullet}}[_{{\cal Q}^{\rm ex}_{\bullet}}, 
\Omega^{\bullet}_{]{{X}_{\bullet}}[^{}_{{\cal Q}^{\rm ex}_{\bullet}}}\! \! < \! {\mathscr M}_{\bullet} \! >), 
$$
where $\Omega^{\bullet}_{]{{X}_{\bullet}}[^{}_{{\cal Q}^{\rm ex}_{\bullet}}}< \! {\mathscr M}_{\bullet} \! >$ 
indicates the restriction of the log differential module of the generic fiber of 
$({\cal Q}^{ \rm ex}_{\bullet}, \mathscr M_\bullet)$ to $]{{X}_{\bullet}}[^{}_{{\cal Q}^{\rm ex}_{\bullet}}$. 
We may then write \eqref{supportrigid5} as
\begin{equation*}
H^m_{X_s,{\rm rig}}(X) \iso
\R ^m \Gamma( {] { X}_{\bullet}[^{}_{{\cal Q}^{\rm ex}_{\bullet}}}, 
[\Omega^{\bullet}_{]{X}_{\bullet}[^{}_{{\cal Q}^{\rm ex}_{\bullet}}}
\rightarrow
\Omega^{\bullet}_{] { X}_{\bullet}[^{}_{{\cal Q}^{\rm ex}_{\bullet}}}
\! \! < \! {\mathscr M}_{\bullet} \! >])
\end{equation*}

\medskip
$$\ast \ast \ast$$

Now we would like to continue our  interpretation. 
As a matter of fact we are going 
to use another {\it exact embedding system} 
to calculate the cohomology of the complex 
$[\Omega^{\bullet}_{]{X}_{\bullet}[^{}_{{\cal Q}^{\rm ex}_{\bullet}}}
\rightarrow
\Omega^{\bullet}_{] { X}_{\bullet}[^{}_{{\cal Q}^{\rm ex}_{\bullet}}}
\! \! < \! {\mathscr M}_{\bullet} \! >]$. 
Indeed we will use a new good
embedding system $({\widetilde {\cal Q}}_{\bullet}^{ \rm ex}, \widetilde{\mathscr M}_\bullet$) 
(exact and smooth as before) admitting a map ${\widetilde {\cal Q}}_{\bullet}^{ \rm ex} \rightarrow \mathscr C$  which 
 is log-smooth over $ (\mathscr C, {\mathscr N})$, {\it i.e.}, we take
into account that $(X, M)$ is log-smooth over $(C, N)$. 
From now, by shrinking $C$, we fix a Frobenius $\sigma_{\mathscr C}$ on $(\mathscr C, \mathscr N)$ 
which is defined by $\sigma_{\mathscr C}(t) = t^p$, where $t$ is a coordinate at $\hat{s}$ in $\mathscr C$ over $\mathcal V$.

\begin{prp} \label{emb2} In the previous notation, it is possible to find a good embedding system 
$X_{\bullet} \rightarrow  {\widetilde {\cal Q}}^{ \rm ex}_{\bullet}$ which fits in the commutative diagram

\begin{equation*} 
\xymatrix{
(X_{\bullet}, M_{\bullet}) \ar[d] \ar[r]^{{\tilde {i}}^{\rm ex}_{\bullet}}& 
({\widetilde{\cal Q}}^{\rm ex}_{\bullet}, {\widetilde {\mathscr M}}_{\bullet})  \ar[d] \\
 (C, N) \ar[r] & (\mathscr C, {\mathscr N})}
\end{equation*}

\noindent
where $\widetilde{\mathcal Q}^{\mathrm{ex}}_\bullet$ 
is smooth and separated of finite type over $\mathcal V$, 
the horizontal maps are exact closed immersions, and 
$(\widetilde{\mathcal Q}^{\mathrm{ex}}_\bullet, \widetilde{\mathscr M}_\bullet)$ is 
log-smooth over $(\mathscr C, \mathscr N)$ such that $\widetilde{\mathscr M}_\bullet$ 
comes from a log-structure on the Zariski topology of $\widetilde{\mathcal Q}^{\mathrm{ex}}_\bullet$. 
Moreover  $(\widetilde{\mathcal Q}^{\mathrm{ex}}_\bullet, \widetilde{\mathscr M}_\bullet)$ admits a lift $\widetilde{\sigma}_\bullet$ 
of Frobenius which is compatible with the Frobenius $\sigma_{\mathscr C}$ on $(\mathscr C, \mathscr N)$.
\end{prp}

\begin{proof}  
Let us keep the notations as in the proof of Proposition \ref{emb}. 
If we consider the $p$-adically complete fiber product 
$({\cal Q}_\bullet ^{\rm ex}, \mathscr M_\bullet) \times (\mathscr C, {\mathscr N})$ 
over ${\mathcal V}$, then the natural morphism 
$(X_\bullet, M_\bullet) \rightarrow ({\cal Q}_\bullet^{\rm ex}, \mathscr M_\bullet) \times (\mathscr C, {\mathscr N})$ 
is a closed immersion, but is not exact. We will need to modify it in order to get an exact one. 
Note that the log-structure $({\cal Q}_\bullet^{\rm ex}, \mathscr M_\bullet) \times (\mathscr C, {\mathscr N})$ is 
the associated log-structure to the monoid $L_\bullet^{\mathrm{ex}} \oplus \mathbb N$. 

We define a co-simplicial 
monoid $\widetilde{L}_\bullet^{\mathrm{ex}}$ by 
$$
      \begin{array}{l}
      \widetilde{L}_m^{\mathrm{ex}} = \prod_{\alpha \in I_m}\widetilde{L}_{m, \alpha}^{\mathrm{ex}}, \\
      \widetilde{L}_{m, \alpha}^{\mathrm{ex}} = 
      L_{m, \alpha}^{\mathrm{ex}} \oplus \mathbb N[\pm((1_{m, \alpha, 1} + \cdots + 1_{m, \alpha, r_{m, \alpha}}) - 1)]\, \, \, 
      \mbox{in}\, \, \, (L_{m, \alpha}^{\mathrm{ex}} \times \mathbb N)^{\mathrm{gr}}, 
      \end{array}
$$
where $1$ is the generator of the last component $\mathbb N$. Indeed, 
if $\pi_\xi(X_{m, \alpha}) \subset X_{m', \alpha'}$ for $\xi : [m] \rightarrow [m']$, 
the image of the element $\pm((1_{m, \alpha, 1} + \cdots + 1_{m, \alpha, r_{m, \alpha}}) - 1)$ 
is contained in $\widetilde{L}_{m', \alpha'}^{\mathrm{ex}}$ by the condition (V) and Lemma \ref{ehc1}. 
Hence the collection $\{ \widetilde{L}_m^{\mathrm{ex}}\, |\, m \geq 0 \}$ forms a co-simplicial monoid. 
Since there is a unique unit $u_\alpha$ on $X_{m, \alpha}$ 
such that $x_{m, \alpha, 1}x_{m, \alpha, 2}\cdots x_{m, \alpha, r_\alpha} = u_\alpha t$, 
the collection $\{ \widetilde{L}_\bullet^{\mathrm{ex}} \rightarrow M_m \}$ of homomorphisms 
induces a homomorphism $\widetilde{L}_\bullet^{\mathrm{ex}} \rightarrow M_\bullet$
of co-simplicial monoids. Moreover, 
the associated log-structure on $X_\bullet$ to the homomorphism 
$\widetilde{L}_\bullet^{\mathrm{ex}} \rightarrow \Gamma(X_\bullet, {\cal O}_{X_\bullet})$ 
of monoids is the given log-structure $M_\bullet$. 

We define a simplicial formal log-scheme $(\widetilde{{\cal Q}}_\bullet^{\mathrm{ex}}, \widetilde{\mathscr M}_\bullet)$ over ${\cal V}$ 
by 
$$ \widetilde{{\cal Q}}_\bullet^{\mathrm{ex}} 
      = ({\cal Q}_\bullet^{\rm ex} \times \mathscr C) 
      \times_{\mathrm{Spf}\, \mathbb Z[L_\bullet^{\mathrm{ex}} \times \mathbb N]\widehat{\, }} 
      \mathrm{Spf}\, \mathbb Z[\widetilde{L}_\bullet^{\mathrm{ex}}]\widehat{\, } $$
       and 
      $\widetilde{\mathscr M}_\bullet $ is the log-structure  associated to the natural homomorphism $ 
     \widetilde{L}_\bullet^{\mathrm{ex}} \rightarrow \Gamma(\widetilde{{\cal Q}}_\bullet^{\mathrm{ex}}, {\cal O}_{\widetilde{{\cal Q}}_\bullet^{\mathrm{ex}}})$.
By the similar proof of Lemma \ref{ehc4} we have 
 
\begin{lmm}  \label{emb33}
\begin{list}{}{}
\item[\mbox{\rm (1)}] There is a natural commutative diagram 
$$
       \begin{array}{ccc}
       & &(\widetilde{{\cal Q}}_\bullet^{\mathrm{ex}}, \widetilde{\mathscr M}_\bullet) \\
       &\nearrow &\downarrow \\
           (X_\bullet, M_\bullet) &\rightarrow &({\cal Q}_\bullet^{\mathrm{ex}}, \mathscr L_\bullet) 
           \times (\mathscr C, {\mathscr N}) \\
           \downarrow& &\hspace*{18mm} \downarrow \mbox{\rm 2-nd proj.}\\
           (C, N) &\rightarrow &(\mathscr C, {\mathscr N})
       \end{array}
$$
of formal log-schemes over ${\cal V}$ such that 
$(X_\bullet, M_\bullet) \rightarrow (\widetilde{{\cal Q}}_\bullet^{\mathrm{ex}}, \widetilde{\mathscr M}_\bullet)$ 
is an exact closed immersion and 
$(\widetilde{{\cal Q}}_\bullet^{\mathrm{ex}}, \widetilde{\mathscr M}_\bullet) \rightarrow 
({\cal Q}_\bullet^{\mathrm{ex}}, \mathscr L_\bullet) \times (\mathscr C, {\mathscr N})$ is log-\'etale. 
In particular, each level  of $(\widetilde{{\cal Q}}_\bullet^{\mathrm{ex}}, \widetilde{\mathscr M}_\bullet)$ 
is log-smooth over $(\mathscr C, {\mathscr N})$. 
\item[\mbox{\rm (2)}] The underlying simplicial formal scheme 
$\widetilde{{\cal Q}}_\bullet^{\mathrm{ex}}$ is smooth over ${\cal V}$.  
\item[\mbox{\rm (3)}] The morphism $\sigma_\bullet \times \sigma_{\mathscr C}$ on 
$({\cal Q}_\bullet^{\mathrm{ex}}, \mathscr M_\bullet) \times (\mathscr C, {\mathscr N})$ 
can be extended  to a morphism 
$\widetilde{\sigma}_\bullet : (\widetilde{{\cal Q}}_\bullet^{\mathrm{ex}}, \widetilde{\mathscr M}_\bullet) 
\rightarrow (\widetilde{{\cal Q}}_\bullet^{\mathrm{ex}}, \widetilde{\mathscr M}_\bullet)$ 
such that $\widetilde{\sigma}_\bullet$ is compatible with the Frobenius $\sigma$ on ${\cal V}$ 
and acts by multiplying $p$ on $\widetilde{L}_\bullet^{\mathrm{ex}}$. 
\end{list}
\end{lmm}

Now the proof of  Proposition \ref{emb2} is complete. 
\end{proof}

Hence we may use the exact embedding of the previous proposition to calculate the  simple complex appearing in  \eqref{supportrigid5}.
Note that the log-tubes  here coincide with the usual ones: 
${]{X}_{\bullet}[^{\rm log}_{{\widetilde{\cal Q}}_{\bullet }}}= {] { X}_{\bullet}[_{{\widetilde{\cal Q}}^{\rm ex}_{\bullet }}}$. 
This is true because the immersion is exact and we have

\begin{equation}
\label{supportrigid6.tris}
H^m_{X_s,{\rm rig}}(X) \iso
\R^m \Gamma({]{X}_{\bullet}[_{{\widetilde{\cal Q}}^{\rm ex}_{\bullet}}}, 
[\Omega^{\bullet}_{]{X}_{\bullet}[_{{\widetilde{\cal Q}}^{\rm ex}_{\bullet}}}
\rightarrow
\Omega^{\bullet}_{]{{X}_{\bullet}}[_{{\widetilde{\cal Q}}^{\rm ex}_{\bullet}}} \! \! < {\widetilde{\mathscr M}_\bullet} >]).
\end{equation}

\noindent 
But we can also introduce an admissible covering of $ {]{X}_{\bullet}[_{{\widetilde{\cal Q}}^{\rm ex}_{\bullet }}}$:  
this is given by the inverse image  of the admissible covering of ${\mathscr C}_K$ given by the tube of  $\{s \}$ (open unit disk hence quasi-Stein) 
and $V$ where $V$ is a strict  affinoid neighborhood of $C \setminus \{ s \}$ in $\mathscr C$. 
The inverse image of such a covering  gives  an admissible covering of ${]{X}_{\bullet}[_{{\widetilde{\cal Q}}^{\rm ex}_{\bullet }}}$:  
${]{X}_{s,\bullet}[_{{\widetilde{\cal Q}}^{\rm ex}_{\bullet }}}$ and $V_{\bullet}$ which are respectively  
the inverse image of the tube of $\{ s \}$ and of $V$  
in ${]{X}_{\bullet}[_{{\widetilde{\cal Q}}^{\rm ex}_{\bullet }}}$.  In particular  the open immersion 
  $\iota_{\bullet K} : {]{X}_{s,\bullet}[_{{\widetilde{\cal Q}}^{\rm ex}_{\bullet}}} 
 \rightarrow {]{X}_{\bullet}[_{{\widetilde{\cal Q}}^{\rm ex}_{\bullet}}}$ is a quasi-Stein  map.
The restrictions to $V_{\bullet}$ of  the two complexes 
which appear in \eqref{supportrigid6.tris} are the same: using this fact 
and $\mathbb R\iota_{\bullet K \ast} = \iota_{\bullet K \ast}$ 
for coherent sheaves by the quasi-Stein property, we can again re-write  \eqref{supportrigid6.tris}. 
In fact, $H^i_{X_s, {\rm rig}}(X)$ can be calculated as the derived functors of the global section functor 
on ${]{X}_{s,\bullet}[_{{\widetilde {\cal Q}^{\rm ex}}_{\bullet}}}$ of

\begin{equation}
\label{supportrigid7}
[\Omega^{\bullet}_{] { X}_{s,\bullet}[_{{\widetilde {\cal Q}^{\rm ex}}_{\bullet}}}
\rightarrow
\Omega^{\bullet}_{] {{ X}_{s,\bullet}}[_{{\widetilde{\cal Q}^{\rm ex}}_{\bullet}}} \! \! <{\widetilde{\mathscr M}_\bullet}>]. 
\end{equation}

\noindent
By our hypotheses $X_s$ is proper and ${\widetilde{\cal Q}}_{\bullet}^{\rm ex}$ is smooth,  
the first complex  calculates the rigid cohomology of $X_s$, 
while the second calculates the log-convergent cohomology of $X_s$ endowed with the induced log-structure from $X$,  
{\it i.e.}, the log-convergent cohomology of $(X_s,M_s)$  with respect to the trivial 
log-structure on the base field, $ H^{m}_{ \rm log-conv}((X_s,M_s)/{\cal V})$.

\vspace*{3mm}

\begin{rmk} \label{tub} It is tempting to  give a name to the hypercohomology 
of the global sections functor of the previous simple complex  \eqref{supportrigid7},
and to denote  it by 
$H^{m}_{X_s, {\rm rig}}({\hat X}/K)$, where $\hat X$ is the completion of $X$ along $X_s$. 
\end{rmk}

As a corollary of these two interpretations, we obtain  the following long exact sequence 
of finite dimensional $K$-vector spaces:
$$
\cdots \longrightarrow H^{m}_{X_s, {\rm rig}}(X)\,
\displaystyle{\mathop{\longrightarrow}^\alpha}\, H^{m}_{\rm rig}(X_s)\, \longrightarrow\, 
H^{m}_{ \rm log-conv}((X_s,M_s)/{\cal V})\, \longrightarrow \cdots. 
 \leqno \eqref{seqsupp}
$$
\noindent
This long exact sequence is compatible with Frobenius. 
By classical results we know that the  rigid terms of such a long exact sequence have Frobenius structures 
(they have a structure of mixed $F$-isocrystals \cite{ch}), 
hence we may endow $H^{m}_{\rm  log-conv}((X_s,M_s)/{\cal V})$ 
with a Frobenius structure (even if $(X_s, M_s)$ is 
not log-smooth over $k$ endowed with the trivial log-structure).

\vspace*{3mm}

\medskip
\begin{rmk} 
\begin{list}{}{}
\item[(1)] We should remark that we could have proved the existence of  \eqref{seqsupp} by just using the exact embedding system $ (X_{\bullet}, M_{\bullet}) \rightarrow ({\cal Q}^{\rm ex}_{\bullet}, {\mathscr M}_
{\bullet})$.  In this case, $V_{\bullet}$, of the covering above,  would have been  constructed as the complement of  tube of $X_{s,\bullet}$  of radius $\eta<1$.
\item[(2)]  More generally if we try  to deal with a general Hyodo-Kato embedding system $(X_\bullet, M_\bullet) \rightarrow ({\mathscr P}_\bullet, \mathscr M_\bullet)$ as in \eqref{HKsystem}, 
then we would have replaced  the complex  in \eqref{supportrigid6.tris} by 
\begin{equation*}
[\Omega^{\bullet}_{]{X}_{ \bullet}[_{{\mathscr P}_{\bullet}}}
\rightarrow
h_{K,  \bullet \ast}\Omega^{\bullet}_{]{{X}_{ \bullet}}[_{{\mathscr P}_{\bullet}}^{\mathrm{log}}} \! \! 
< \mathscr M_\bullet >]
\end{equation*}
where 
$h_{K,  \bullet} : ]{{X}_{ \bullet}}[_{{\mathscr P}_{\bullet}}^{\mathrm{log}} 
\rightarrow ]{{X}_{ \bullet}}[_{{\mathscr P}_{\bullet}}$ 
is the canonical morphism. As a corollary of our local global comparison, we have an isomorphism 
$$
      H^{m}_{X_s, {\rm rig}}(X) \iso 
      \R^m\Gamma(]{{X}_{s, \bullet}}[_{{\mathscr P}_{\bullet}}, 
      [\Omega^{\bullet}_{]{X}_{s, \bullet}[_{{\mathscr P}_{\bullet}}}
\rightarrow
h_{K, s, \bullet \ast}\Omega^{\bullet}_{]{{X}_{s, \bullet}}[_{{\mathscr P}_{\bullet}}^{\mathrm{log}}} \! \! 
< \mathscr M_\bullet >])
$$
where 
$h_{K, s, \bullet} : ]{{X}_{s, \bullet}}[_{{\mathscr P}_{\bullet}}^{\mathrm{log}} 
\rightarrow ]{{X}_{s, \bullet}}[_{{\mathscr P}_{\bullet}}$ is the restriction of $h_{K, \bullet}$. 
However, one can not directly apply our argument of global and local 
comparison to this complex, for the reason that 
$h_{K, \bullet}^{-1}(V_\bullet)$ is not isomorphic to $V_\bullet$ in general.
As a matter of fact,  the local existence of exact embeddings is sufficient  for proving  
 the exact sequence \eqref{seqsupp}. The authors, however, thought  that 
it  worthwhile  to prove the existence of global exact embedding systems 
of Propositions \ref{emb}, \ref{emb2} for use in  further investigations. 
\end{list}
\end{rmk}

\medskip

$$
\ast \ast \ast
$$

By putting together \eqref{seqsupp} and \eqref{rel6} we have the sequence (as in the introduction):
$$
\begin{array}{ll}
\cdots &\rightarrow H^m_{\rm rig}(X_s) \xrightarrow{\gamma} 
H^m_{\rm log-crys}((X_s,M_s)/{\cal V}^\times) \otimes K  \xrightarrow{N_m}  H^m_{\rm log-crys}((X_s,M_s)/{\cal V}^\times) \otimes K(-1) \xrightarrow{\delta} H^{m+2}_{X_s,{\rm rig}}(X) \\
&\xrightarrow{\alpha}  H^{m+2}_{\rm rig}(X_s) \rightarrow \cdots. 
\end{array} \leqno{(\ref{sequenceA})}
$$
where the maps $\gamma$ and $\delta$ are defined from \eqref{seqsupp} and \eqref{rel6}  by the composites
$$
H^m_{\rm rig}(X_s) \rightarrow  H^m_{\rm log-conv}((X_s, M_s)/{\cal V}) \rightarrow 
H^m_{\rm log-crys}((X_s,M_s)/{\cal V}^\times) \otimes K
$$
and 
$$
H^m_{\rm log-crys}((X_s,M_s)/{\cal V}^\times) \otimes K (-1) \rightarrow  H^{m+1}_{\rm log-conv}((X_s, M_s)/{\cal V}) 
\rightarrow H^{m+2}_{X_s,{\rm rig}}(X).
$$

\medskip

\bigskip

\section{Exactness of the sequence}

 In this last part we would like to prove the exactness of the previous long  sequence (as in the introduction):
$$
\begin{array}{ll}
\cdots &\rightarrow H^m_{\rm rig}(X_s) \xrightarrow{\gamma} 
H^m_{\rm log-crys}((X_s,M_s)/{\cal V}^\times) \otimes K  \xrightarrow{N_m} 
H^m_{\rm log-crys}((X_s,M_s)/{\cal V}^\times) \otimes K(-1) \xrightarrow{\delta} H^{m+2}_{X_s,{\rm rig}}(X) \\
&\xrightarrow{\alpha}  H^{m+2}_{\rm rig}(X_s) \rightarrow \cdots
\end{array} \leqno{(\ref{sequenceA})}
$$
and we will make the further  hypothesis that the field $k$ is finite with  $q=p^a$ elements. 
We will show how such a result (after our translation of the topological tools in our framework) 
is a formal corollary of the fact that in characteristic  $p$ the monodromy filtration coincides with the weight filtration. 
We recall that in the third section we  had the following interpretation \eqref{concl1}:

\begin{equation*}
(\R^mf_{(X,M)/(\mathscr C, {\mathscr N}), {\rm an} \ast}({\O}_{{\rm an},X,K})_{\mid] s[_{\mathscr C}})_{\hat s_K} 
\iso H^m_{\rm log-crys}((X_s,M_s)/{\cal V}^\times) \otimes K, 
\end{equation*}

\noindent 
where the tube $] s[_{\mathscr C}$ is isomorphic, by a choice of the lift $\hat s$ of $s$, 
to the open unit disk with a parameter $t$ corresponding to $\hat s_K$. 
The log-analytic cohomology can be represented by the trivialization \eqref{residue}:
$$
    (\R^mf_{(X,M)/(\mathscr C, {\mathscr N}){\mid] s[_{\mathscr C}}, {\rm an} \ast}({\O}_{{\rm an},X,K}), \nabla, \varphi_m) 
    \iso (V_m \otimes \O_{] s[_{\mathscr C}}, N_m \otimes 1 + 1 \otimes d, F_m \otimes \sigma_{\mathscr C} ). 
$$
where $V_m$ is a $K$-vector space endowed with a monodromy endomorphism $N_m$ 
and a Frobenius structure $F_m$ such that $qF^a_mN_m=N_mF^a_m$. Moreover, 
$V_m$ is isomorphic to $H^m_{\rm log-crys}((X_s,M_s)/{\cal V}^\times) \otimes K$ by Therorem \ref{monodromy}. 

Associated with a nilpotent operator we have a {\it monodromy} filtration on $V_m$. 
We will indicate it by ${\cal M}$ 
and the steps by ${\cal M}_{j}$, $j \in \Z$. Moreover the Frobenius $F^a_m$ 
induces a filtration called the   {\it weight} filtration. 
In this setting we have the equivalence between the two filtrations along the line of  Crew's proof \cite{crew} 10.8:

\begin{thr}
 \label{crew}  
 Under the previous hypotheses, the log-analytic cohomology  sheaf 
  $$\R ^m f_{(X,M) / (\mathscr C, {\mathscr N}), {\rm an} \ast } ({\O}_{{\rm an},X,K})$$
on ${\mathscr C}_K \setminus ]s[_{\mathscr C}$ is pure of weight $m$. 
Moreover  $\R^mf_{(X,M) / (\mathscr C, {\mathscr N}), {\rm an} \ast } ({\O}_{{\rm an}, X,K})$ 
on $] s[_{\mathscr C}$ is unipotent and, for each $j$, 
the graded part for the monodromy filtration, $gr_j^{\cal M} V_m$, is pure of  weight $m+j$ 
for the Frobenius action.
\end{thr}
\begin{proof} In view of our hypotheses the family $X \rightarrow C$ was given by a 
proper and log-smooth morphism with only one  (classically) singular fiber at $s$. 
Since $\R^mf_{(X,M) / (\mathscr C, {\mathscr N}), {\rm an} \ast } ({\O}_{{\rm an}, X,K})$ 
is locally free by Theorem \ref{relan}, 
$\R^mf_{(X,M) / (\mathscr C, {\mathscr N}), {\rm an} \ast } ({\O}_{{\rm an}, X,K})$ 
calculates the rigid cohomology of the proper and smooth fibers except $X_s$. 
Hence they are pure for the Frobenius structure. Now we can apply Crew's theorem, {\it ibid.}
\end{proof}

\noindent 
This gives an equivalence  between monodromy and weight filtrations (up to  a shift). 

\begin{crl}  \label{weight} On  $H^m_{\rm log-crys}((X_s,M_s)/{\cal V}^\times) \otimes K$ the $\mathrm{Ker}\, N_m$ 
has (Frobenius) weights less  than or equal to $m$. 
While $H^m_{\rm log-crys}((X_s,M_s)/{\cal V}^\times) \otimes K (-1)/\mathrm{Im}\, N_m$ has 
(Frobenius) weights greater  than or equal to $m+2$.
\end{crl}
\begin{proof} By Deligne \cite{de}1.6, we know that $N_m$ is injective on $V_m/{\cal M}_0 \rightarrow V_m/{\cal M}_{-2}(-1)$, 
hence $\mathrm{Ker}\, N_m \subset {\cal M}_0$. 
On the other hand for each $j<0$ , $N_m{\cal M}_{j+2}={\cal M}_{j}(-1)$: 
hence $ {\cal M}_{-1} (-1)\subset \mathrm{Im}\, N_m$. We conclude that  $V_m(-1)/\mathrm{Im}\, N_m$ 
has weights bigger than or equal to $m+2$.
\end{proof}

 \begin{crl}  \label{direct-sum}  The Frobenius induces a structure of mixed  $F^a$-isocrystal 
 on $H^m_{\rm log-conv}((X_s, M_s)/{\cal V})$ {\rm (\cite{ch})}. Moreover we have a direct sum decomposition 
 with respect to the weights
 $$
H^m_{\rm log-conv}((X_s, M_s)/{\cal V})= H^m_{\rm log-conv}((X_s, M_s)/{\cal V})^{\leq m} 
\oplus H^m_{\rm log-conv}((X_s, M_s)/{\cal V})^{>m}.
$$
\end{crl}
 
 \begin{proof} The cohomological  groups $H^m_{\rm log-conv}((X_s, M_s)/{\cal V})$  sit  in two long exact sequences  \eqref{seqsupp} and \eqref{rel6}  both compatible with Frobenius:  hence they are mixed as $F^a$-isocrystals. 
 In \eqref{seqsupp},  by  the theory of "classical" rigid cohomology we know that $H^m_{\rm rig}(X_s)$ 
 has weights $\leq  \!  \! m$  (\cite{cl} 2.2) while  $H^{m+1}_{X_s,{\rm rig}}(X)$ has weights strictly bigger than $m$ (\cite{ch}, 2.3): 
 then one can insert   $H^m_{\rm log-conv}((X_s, M_s)/{\cal V})$ in a short exact sequence where the first non trivial term 
 has weights $\leq \!    m$ (in fact  it is a quotient of $H_{\rm rig}^m(X_s)$), while the last non trivial term has weights $> \! m$ 
 (because it is a sub $F^a$-isocrystal of $H^{m+1}_{X_s, {\rm rig}}(X)$).  But, using \eqref{rel6}, 
 we can insert $H^m_{\rm log-conv}((X_s, M_s)/{\cal V})$  in another short exact sequence : 
 this time the first non trivial term is the quotient  
 $H^{m-1}_{\rm log-crys}((X_s,M_s)/{\cal V}^\times) \otimes K (-1)/\mathrm{Im}\, N_{m-1}$ 
 which has weights $> \! m$ by Corollary \ref{weight} and the last term is $\mathrm{Ker}\, N_m$ 
 on  $H^m_{\rm log-crys}((X_s,M_s)/{\cal V}^\times) \otimes K$ 
 which has weights $\leq  \! m$ (again by Corollary \ref{weight}).   
 \end{proof}

We are now ready for the last step in the proof of exactness of the Clemens-Schmid sequence. 
As we said, the   part of the long exact sequence  \eqref{seqsupp}
$$
 \dots \rightarrow  H^{m}_{\rm rig}(X_s) 
 \rightarrow H^{m}_{ \rm log-conv}((X_s,M_s)/{\cal V}) \rightarrow H^{m+1}_{X_s, {\rm rig}}(X) \rightarrow \cdots
$$
is compatible with the Frobenius structure, hence it gives a surjection

\begin{equation*} 
  H^{m}_{\rm rig}(X_s)= H^{m}_{\rm rig}(X_s)^{\leq m} \rightarrow H^{m}_{ \rm log-conv}((X_s,M_s)/{\cal V})^{\leq m}.
 \end{equation*}
 
 \noindent
We then have  $\mathrm{Im}(H^{m}_{\rm rig}(X_s)\rightarrow 
H^m_{\rm log-conv}((X_s,M_s)/{\cal V}))= H^{m}_{\rm log-conv}((X_s,M_s)/{\cal V})^{\leq m}$ 
and of course it is contained in $\mathrm{Ker}\, N_m$ according to the long exact sequence \eqref{rel6}.  
Since $\mathrm{Ker}\, N_m$ has weights at most  $m$  and the $\mathrm{Coker}\, N_{m-1}$ has weights $\geq m+1$, 
then the kernel is exactly  isomorphic to $ H^{m}_{\rm  log-conv}((X_s,M_s)/{\cal V} )^{\leq m}$ under the map in \eqref{rel6}. 
Hence  it is  $\mathrm{Im}(H^{m}_{\rm rig}(X_s) \rightarrow H^m_{\rm log-crys}((X_s,M)/{\cal V}^\times) \otimes K)$ 
in $\gamma$ of  \eqref{sequenceA}. In the second part of the sequence, we know that $\mathrm{Coker}\, N_m$ 
has weights $\geq m+2$ (because of the Frobenius twist).  
And it is isomorphic to $H^{m+1}_{\rm  log-conv}(X_s,M_s)/{\cal V} )^{\geq m+2}$ 
because  $\mathrm{Ker}\, N_{m+1}$ has weights $\leq m+1$. Consider 
the long exact sequence \eqref{seqsupp}

$$
H^{m+1}_{\rm  log-conv}((X_s,M_s)/{\cal V} ) 
 \rightarrow H^{m+2}_{X_s,{\rm  rig}}(X) \rightarrow H^{m+2}_{\rm rig}(X_s).
$$

\noindent 
The kernel of the last map is the isomorphic to   $H^{m+1}_{\rm log-conv}((X_s,M_s)/{\cal V})^{\geq m+2}$ 
because $ H^{m+1}_{\rm rig}(X_s)$ has weights $\leq m+1$ and  $H^{m+2}_{X_s, {\rm rig}}(X)$ has weights $\geq m+2$. 
 
This concludes the proof of the exactness of the Clemens-Schmid sequence.

	Bruno Chiarellotto, Dip. Matematica, Univ. Padova, Via Trieste 63, 35121 Padova (Italy), email: chiarbru@math.unipd.it
	
	Nobuo Tsuzuki, Mathematical Inst. Tohoku Univ.,  6-3, Aoba, Aramaki, Aoba-Ku, Sendai 980-8578, email: tsuzuki@math.tohoku.ac.jp

\addcontentsline{toc}{section}{References}
	\begin{thebibliography}{MVW06}
	
	\bibitem[Be74]{Ber:Cri74}
	P. Berthelot,
	\newblock Cohomologie cristalline des sch\'emas de caract\'eristique $p>0$, 
	\newblock {\em Lecture Notes in Mathematics}, Vol. 407, Berlin, New York, Springer-Verlag. 
	
	
	\bibitem[Be97]{Ber:Dua97}
	P. Berthelot,
	\newblock Dualit{\'e} de {P}oincar{\'e} et formule de {K}{\"u}nneth en
	  cohomologie rigide.
	\newblock {\em C. R.  Math. Acad. Sci. Paris} S{\'e}r. I Math. 325(1997), 493-498.
	
	\bibitem[BO78]{BO78}
	P. Berthelot, A. Ogus
	\newblock Note on crystalline cohomology.
	\newblock {\em Princeton University Press and University of Tokyo Press}, 1978.
	
	\bibitem[CH98]{ch}
	B. Chiarellotto,
	\newblock Weights in rigid cohomology: applications to unipotent $F$-isocrystals. 
	\newblock   {\em Ann. Sci. \'Ec. Norm.Sup\'er.}  31(1998), 683-715.
	
	\bibitem[CL98]{cl}
	B. Chiarellotto, B. Le Stum,
	\newblock Sur la puret\'e de la cohomologie cristalline.
	\newblock  {\em C. R.  Math. Acad. Sci. Paris} S\'er.I Math. 326(1998), 961-963.
	
	\bibitem[CT03]{ct}
	B. Chiarellotto, N. Tsuzuki, 
	\newblock Cohomological Descent of  Rigid Cohomology for \'etale coverings. 
	\newblock   {\em Rend. Semin. Mat.  Univ. Padova} 109(2003), 63-215.
	
	\bibitem[CH84]{chr} G. Christol, 
	\newblock Un th\'eor\`eme de transfert pour les disques singuliers r\'eguliers, 
	\newblock   {\em Cohomologie $p$-adique, Ast\'erisque (1984), no. 119-120, 151-168.}
	
	
	\bibitem[CL77]{cle} C.H. Clemens, 
	\newblock Degeneration of K\"ahler manifolds.  
	\newblock {\em Duke Math. J.}  44(1977), no. 2, 215-290.
	

	
	
	\bibitem[CR98]{crew}
	R. Crew,
		\newblock Finiteness theorems for the  cohomology of an overconvergent isocrystal on a curve.
	\newblock   {\em Ann. Sci. \'Ec. Norm.Sup\'er.} 31(1998), 717-763.
	
	
	\bibitem[DE80]{de}
	P. Deligne,
	\newblock La conjecture de Weil II.
	\newblock {\em Publ. Math. Inst. Hautes \'Etudes Sci.} 52(1980), 137-252. 

	
	\bibitem[HK94]{hy-ka}
	O. Hyodo, K. Kato,	\newblock Semi stable reduction and crystalline cohomology with logarithmic poles.  
	\newblock {\em P\'eriode p-adiques}, (Bures-sur-Yvette 1988) Ast\'erisque, 223(1994), 221-268.

	
	\bibitem[KA89]{Ka:log}
	K. Kato,
	\newblock Logarithmic structures  of Fontaine-Illusie.
	\newblock  {\em Algebraic analysis, geometry and number theory}, the Johns Hopkins Univ. Press (1989), 191-224.
	
	\bibitem[KO68]{KO}
	N.M. Katz, T. Oda,
	\newblock On the differentiation of De Rham cohomology classes with respect to parameters. 
	\newblock  {\em J. Math. Kyoto Univ.}, 8(1968), 199-213.
	
	\bibitem[IL94]{il}
	L. Illusie,
	\newblock Autour du th\'eor\`eme de monodromie locale.
	\newblock  {\em P\'eriode p-adiques}, (Bures-sur-Yvette 1988) Ast\'erisque 223(1994), 9-57.
	
	
	\bibitem[LAU11]{laud}
	A. G.B. Lauder,
	\newblock Degenerations and limit Frobenius structures in rigid cohomology.
	\newblock  {\em LMS J. Comput. Math.} 14 (2011), 1-33.
	
	\bibitem[LAZ62]{Laz}
	M. Lazard, 
	\newblock Les z\'eros d'une fonction analytique d'unse variable sur corps valu\'e complet. 
	\newblock {\em Publ. Math. Inst. Hautes \'Etudes Sci.} 14(1962), 47-75.	
	
	\bibitem[LS07]{ls}
	B. Le Stum,
	\newblock Rigid cohomology.
	\newblock  {\em Cambridge Tracts in Math.},  172 (2007),  Cambridge Univ. Press.
	
 
	  \bibitem[LE07]{le}
	M. Levine,
	\newblock Motivic tubular neighborhood. 
	\newblock   {\em Doc. Math.} 12(2007),  71-146.
	
	

         \bibitem[MO84]{mo}
	D.R. Morrison,
	\newblock The Clemens-Schmid exact sequence and applications.
	\newblock   {\em Topics in transcendental algebraic geometry} 101-119,   
	Ann. of Math. Stud.  106, Princeton Univ. Press, Princeton NJ 1984. 
	
	\bibitem[NA06]{na} Y. Nakkajima, 
	\newblock Signs in weight spectral sequences, monodromy-weight conjectures, 
	log Hodge symmetry and degenerations of surfaces.  
	\newblock {\em Rend. Sem. Mat. Univ. Padova}  116(2006), 71-185.
	
	\bibitem[PE03]{Pet:Cla03}
	D. Petrequin,
	\newblock Classes de {C}hern et classes de cycles en cohomologie rigide.
	\newblock {\em Bull. Soc. Math. France} 131(2003), 59-121.




	
	
	\bibitem[SH00]{shiho:fund1}
	A. Shiho,
	\newblock Crystalline fundamental groups I-Isocrystals 
	on log crystalline site and log  convergent site.
	\newblock  {\em J. Math. Sci.Univ.Tokyo} 7(2000), 509-656.
	
	\bibitem[SH02]{shiho:fund2}
	A. Shiho,
	\newblock Crystalline fundamental groups II-Log convergent cohomology and rigid cohomology.
	\newblock  {\em J. Math. Sci.Univ.Tokyo} 9(2002), 1-163.
	
	
	
	\bibitem[SH08]{shiho:rel1}
	A. Shiho,
	\newblock Relative log-convergent cohomology and relative rigid cohomology I.
	\newblock arXiv:0707.1742v2 (2008).
	
	\bibitem[SH08A]{shiho:rel2}
	A. Shiho,
	\newblock Relative log-convergent cohomology and relative rigid cohomology II.
	\newblock arXiv:0707.1743v2 (2008).
	
	\bibitem[ST76] {st}
	J. Steenbrink,
	\newblock Limits of Hodge structures.	
	\newblock  {\em Invent. Math.} 31(1976), 229-257.

	
		
	\bibitem[TS04]{ts}
	N. Tsuzuki,
	\newblock Cohomological descent in rigid cohomology. 
	\newblock   {\em Geometric Aspects of Dwork Theory II},   
	Adolphson, Baldassarri, Berthelot, Katz, Loeser  Eds., W. de Gruyter (2004), 931-982.
	
	
	 
	  

	

	\end{thebibliography}
\end{document}